\documentclass[12pt,reqno]{amsart}
\usepackage{amsfonts}
\usepackage{amsmath}
\usepackage{amssymb}
\usepackage{amsthm}
\usepackage[a4paper,scale=0.77]{geometry}
\def\G{\mathcal G}
\def\U{\mathrm U}
\def\S{\mathcal S}
\def\C{\mathbb C}
\def\R{\mathbb R}
\DeclareMathOperator{\trace}{trace}
\DeclareMathOperator{\spec}{spec}
\def\i{\mathrm i}
\def\d{\mathrm d}
\let\Re\relax
\DeclareMathOperator{\Re}{Re}

\DeclareMathOperator{\Op}{Op}

\def\Gh{{\widehat \G}}
\def\rd{{\rho,\delta}}

\def\HS{{\mathtt{HS}}}
\newtheorem{thm}{Theorem}[section]
\newtheorem{prop}[thm]{Proposition}
\newtheorem{lem}[thm]{Lemma}
\newtheorem{cor}[thm]{Corollary}
\numberwithin{equation}{section}
\theoremstyle{definition}
\newtheorem{df}{Definition}[section]
\theoremstyle{remark}
\newtheorem{rem}{Remark}[section]
\begin{document}
%\begin{frontmatter}
%
\title[Global functional calculus on compact Lie groups]
{Global functional calculus for operators on compact Lie groups}
%\author[MR]{Michael Ruzhansky\fnref{fn1}}
%\ead{m.ruzhansky@imperial.ac.uk}
\thanks
%\fntext[fn1]
{The first author was supported by the EPSRC Leadership Fellowship EP/G007233/1 and by EPSRC Grant EP/K039407/1.}

%\address[MR]{
%  Department of Mathematics,
%%  \endgraf
%  Imperial College London,
%%  \endgraf
%  180 Queen's Gate, London SW7 2AZ, UK
%%  \endgraf
%}
%
%
%\author[JW]{Jens Wirth\corref{cor1}}
% \ead{Jens.Wirth@math.lmu.de, jens.wirth@iadm.uni-stuttgart.de}
%\address[JW]{
%% Institut f\"ur Analysis, Dynamik und Modellierung, Universit\"at Stuttgart, Pfaffenwaldring 57, 70569 %Stuttgart, Germany\\
% Mathematisches Institut der Universit\"at M\"unchen,
%%  \endgraf
%  Theresienstr. 39,
%%  \endgraf
%  80333 M\"unchen, Germany
% % \endgraf
% // 
% \\
%  Institut f\"ur Analysis, Dynamik und Modellierung, Universit\"at Stuttgart, Pfaffenwaldring 57, 70569 Stuttgart, Germany 
%}
% \cortext[cor1]{Corresponding author.}

\address{
  Michael Ruzhansky:
  \endgraf
  Department of Mathematics
  \endgraf
  Imperial College London
  \endgraf
  180 Queen's Gate, London SW7 2AZ, UK
  \endgraf
  {\it E-mail address} {\rm m.ruzhansky@imperial.ac.uk}
  \endgraf
  \medskip
  Jens Wirth:
  Institut f\"ur Analysis, Dynamik und Modellierung
  \endgraf
  Universit\"at Stuttgart
  \endgraf
  Pfaffenwaldring 57
  \endgraf 
  70569 Stuttgart, Germany 
%  \endgraf
%  Mathematisches Institut der Universit\"at M\"unchen
%  \endgraf
%  Theresienstr. 39
%  \endgraf
%  D-80333 M\"unchen, Germany
  \endgraf
  {\it E-mail address} {\rm Jens.Wirth@iadm.uni-stuttgart.de}
  }

\date{\today}

\subjclass{Primary 35S05; Secondary 22E30}
\keywords{Functional calculus, pseudo-differential operators, compact Lie groups, G{\aa}rding inequality}

\begin{abstract}
In this paper we develop the functional calculus for elliptic operators on compact Lie 
groups without the assumption that the operator is a classical pseudo-differential operator. 
Consequently, we provide a symbolic descriptions of complex powers of such operators.
As an application, we give a constructive symbolic proof of the G\r{a}rding inequality
for operators in $(\rho,\delta)$-classes in the setting of compact Lie groups.
\end{abstract}

%\begin{keyword}
%Functional calculus \sep pseudo-differential operators \sep compact Lie groups \sep G{\aa}rding inequality
%
%\MSC[2010]{Primary 35S05; Secondary 22E30}
%\end{keyword}

%\end{frontmatter}
\maketitle

\section{Introduction}

In \cite{Seeley}, Seeley has developed the functional calculus of classical 
pseudo-differential operators on compact manifolds. The main idea of the 
construction was to define functions of homogeneous components of
the symbol in local coordinates, and then patch them together obtaining
a globally defined function of an operator.  

Over the years,
this idea has been developed further for elliptic operators in different settings,
see e.g. Kumano-go--Tsutsumi \cite{Kumano-go-Tsutsumi:complex-powers-1973}, 
Beals \cite{Beals:commutators-DMJ-1977}, Kumano-go \cite{Kumano-go:BOOK-pseudos},
Helffer \cite{Helffer:spectrale-Asterisque-1984}, 
Coriasco, Schrohe and Seiler 
\cite{Coriasco-Schrohe-Seiler:complex-powers-MZ-2003},
to mention only very few contributions.
On manifolds with particular geometries (of the manifold itself or of its boundary)
the construction of complex powers of operators becomes adapted to the
underlying geometry, see e.g. Schrohe 
\cite{Schrohe:complex-powers-IEOT-1986,Schrohe:complex-powers-noncompact-MA-1988}
for geometries with fibered boundaries, or Loya
\cite{Loya:complex-powers-resolvents:JFA-2001,
Loya:complex-powers-weighted-Sobolev-JAM-2003}
for manifolds with conical singularities, 
where in the latter the analysis is based on the heat kernel techniques 
\cite{Loya:heat-kernels-IJM-2003}.
There are important applications, such as those 
of the $\zeta$-function of an operator $T$ (defined by 
$\zeta(z)=\trace T^{z}$) to the index theory, or to the  
evolution equations.
There are further applications of complex powers of operators to Wodzicki-type
residues, see e.g. Buzano and Nicola \cite{Buzano-Nicola:complex-powers-JFA-2007}
for a good and more extensive literature review of the above topics, 
as well as for the complex powers in the Weyl--H\"ormander calculus.

In this paper we show that in the setting of compact Lie groups one can work
with functions of operators using the globally defined matrix symbols instead
of representations in local coordinates, which is the version of the analysis well
adopted to the operator theory on compact Lie groups.
These matrix symbols and their calculus have been recently developed in
\cite{ruzhansky+turunen-book,Ruzhansky+Turunen-IMRN}, and in
\cite{Ruzhansky-Turunen-Wirth:arxiv} the characterisation of operators
in H\"ormander's classes $\Psi^{m}=\Psi^{m}_{1,0}$ on 
the compact Lie group viewed as 
a manifold was given in terms of these matrix symbols, thus providing a
link between local and global symbolic calculi.
In Section \ref{SEC:preliminaries} we briefly review the required parts of these
constructions. In particular, notions such as the ellipticity and 
hypoellipticity can be characterised in terms of the matrix symbols.
The matrix symbols have been instrumental in handling other problems, 
for example for proving the H\"ormander--Mikhlin
multiplier theorem \cite{Ruzhansky-Wirth:Lp-FAA} in the 
setting of general compact Lie groups.

In order to approach the functional calculus of operators from the point
of view of the symbolic calculus of matrix symbols, first
we introduce a notion of parameter-dependent
ellipticity in our setting and investigate its properties. 
Consequently, we apply it to defining the functions of matrix symbols which 
are then quantized to provide functions of operators.
If the operators are sufficiently nice, for example self-adjoint, the obtained
functions of operators coincide with those that can be defined by the
spectral calculus. 

Therefore, we can note that the proposed approach applies to a wide class of
operators which, in particular, do not have to be self-adjoint, and do not have to
be classical pseudo-differential operators. In fact, we can work with the families
of $(\rho,\delta)$-classes defined in terms of the matrix symbols. In the case
of $\rho=1$ and $\delta=0$, this class coincides with the usual H\"ormander
class of pseudo-differential operators on compact manifolds, but without
having to assume the existence of homogenous expansion for symbols
in local coordinates.

We note that the standard theory of pseudo-differential operators of type
$(\rho,\delta)$ on manifolds requires certain relations between $\rho$ and $\delta$,
most usually it is the requirement that 
\begin{equation}\label{EQ:rd}
1-\rho\leq\delta<\rho,
\end{equation}
see e.g. Shubin \cite[Section 4]{Shubin:BK-1987}, implying, in particular,
that $\rho>\frac12$. However, when working
with matrix symbols, one can usually allow any $0\leq \rho\leq 1$ and
$0\leq\delta\leq 1$. This becomes effective in handling certain classes
of operators, for example resolvent operators for vector fields on a compact
Lie group $\G$ have
symbols in the symbol classes $\S^{0}_{0,0}(\G\times\Gh)$, see
\cite{Ruzhansky-Turunen-Wirth:arxiv} for this and for other examples of the 
appearance of different symbol classes as parametrices for hypoelliptic
operators, which can not be
handled by the standard theory in view of the restriction 
\eqref{EQ:rd}.
Thus, in our case, we do not impose any relations on $0\leq\rho,\delta\leq 1$
when working with functions of symbols, while for treating functions of
operators we ask only for $\delta<\rho$. 

We give two applications of the developed functional calculus, more precisely,
of the possibility of taking square roots of positive symbols and operators.
First, we prove the $L^{2}$ and Sobolev-$L^{2}$ boundedness of operators
with symbols in the class $\S^{0}_{\rd}(\G\times\Gh)$ for any 
$0\leq\delta<\rho\leq 1$. In particular, it provides a criterion
for the boundedness of operators in the setting of compact
Lie groups beyond the condition \eqref{EQ:rd}.

Second, we prove the G{\aa}rding inequality for operators under the same
relation $0\leq\delta<\rho\leq 1$. While the G{\aa}rding inequality for 
pseudo-differential operators is well-known on $\mathbb R^n$ for the same range
of $\rho$ and $\delta$ (see e.g. \cite{Taylor:BOOK-pseudos}),
and also on manifolds under the condition
\eqref{EQ:rd}, on Lie groups it was obtained in \cite{Bratteli-et-al:Garding-on-groups}
for operators in H\"ormander's classes of type $(1,0)$
using Langland's result on semigroups on Lie groups 
\cite{Langlands:holomorphic-semigroups}.
For operators of this type, the same conclusion can be obtained from the
local theory. Moreover, the sharp G{\aa}rding inequality for such operators is 
also known in the setting of compact Lie groups under the positivity condition of
the matrix symbol, see \cite{Ruzhansky+Turunen-JFA-Garding}. In this direction
we can note other existing lower bounds for pseudo-differential operators on compact manifolds, for example
the Melin--H\"ormander inequality, see 
\cite{Mughett-Parenti-Parmeggiani:lower-bounds-CPDE} for currently one of the
most general statements for classical pseudo-differential operators.

Here, the functional calculus will allow us to obtain the G{\aa}rding inequality,
with a completely symbolic and constructive proof for operators of type $(\rd)$ with
$0\leq\delta<\rho\leq 1$. In particular, this applies to parametrices for
hypoelliptic operators from classical H\"ormander classes of type $(1,0)$,
when often the parametrix has the matrix symbol of type $(\rd)$ for some
different $\rd$, see \cite{Ruzhansky-Turunen-Wirth:arxiv} for a number of examples.

The plan of the paper is as follows. In Section \ref{SEC:preliminaries} we review
the necessary elements of the symbolic calculus for matrix symbols that will be
required in the proof. In Section \ref{SEC:ellipticities} we introduce the
notion of parameter-dependent ellipticity that will be crucial for further analysis.
In Section \ref{SEC:functions} we establish the functional calculi for matrix
symbols and for the corresponding operators. In Section \ref{SEC:Sobolev} we prove the 
$L^{2}$ and Sobolev-$L^{2}$ boundedness for operators, and in Section
\ref{SEC:Garding} we apply it together with the functional calculus to
establish the G{\aa}rding inequality. Finally, in Section \ref{SEC:appendix}
we provide an appendix where we prove several technical results used
throughout the paper.

\section{Preliminaries about matrix symbols and their symbolic calculus}
\label{SEC:preliminaries}

Let $\G$ be a compact Lie group and let $\widehat \G$ denote the set of equivalence classes of 
irreducible unitary representations $\xi : \G \to \U(d_\xi)$ of $\G$. It has been shown in 
\cite{ruzhansky+turunen-book}, \cite{Ruzhansky-Turunen-Wirth:arxiv} that the H\"ormander class $\Psi^m_{1,0}(\G)$ of pseudo-differential 
operators on $\G$ is characterised in terms of its global symbols
\begin{equation*}
   \sigma_A(x,\xi) = \xi(x)^{*} (A\xi)(x) 
\end{equation*}
by symbolic estimates
\begin{equation}\label{eq:symbeq}
   \|\partial^\alpha_x \mathbb D^\beta_\xi \sigma_A(x,\xi)\|_{\rm op} \le C_{\alpha,\beta} \langle\xi\rangle^{m-|\beta|}
\end{equation}
in terms of (usual) derivatives $\partial_x$ acting on points of $\G$ and difference operators 
$\mathbb D_\xi$ acting on the representation lattice. This symbolic estimate resembles symbolic 
properties of local H\"ormander symbols defined in terms of local coordinates, 
however, it applies to a global object defined on $\G\times\widehat\G$.  The set of all  symbols satisfying \eqref{eq:symbeq} will be denoted as $\S^m_{1,0}(\G\times\widehat\G)$, we postpone a formal definition to the end of Section \ref{SEC:differences}. First we will recall preliminaries concerning Fourier series, global operator quantization and the associated global calculus.

\subsection{Fourier transform on $\G$ and global symbols}
The Fourier transform on $\G$ is defined in terms of the Peter--Weyl decomposition 
\begin{equation*}
    L^2(\G) = \bigoplus_{[\xi]\in\widehat\G}  \mathcal H_\xi ,\qquad \mathcal H_\xi = \{ x\mapsto\trace(\xi(x) E) : E\in\C^{d_\xi\times d_\xi}\}
\end{equation*}
of $L^2(\G)$ into spaces of matrix coefficients. Orthogonal projections onto the subspaces 
$\mathcal H_\xi$ are given in terms of  convolutions with the characters 
$\chi_\xi(x) =  \trace\xi(x)$ multiplied by the dimension $d_\xi$. 
We identify the spaces $\mathcal H_\xi$
with $\C^{d_\xi\times d_\xi}$ in terms of the matrices $E$ in the above formula and write in short
\begin{equation*}
   \widehat f(\xi) = \int_\G f(x) \xi(x)^* \d x
\end{equation*}
for the Fourier coefficient of a function $f\in L^2(\G)$ at the repesentation $\xi$. The above 
decomposition implies that
\begin{equation*}
    f(x) = \sum_{[\xi]\in\widehat\G} d_\xi \trace \big(\xi(x)\widehat f(\xi) \big) 
\end{equation*}
as orthogonal direct sum and, in particular, it follows that the Parseval identity
\begin{equation}\label{eq:Parseval}
   \| f\|_{L^2(\G)}^2 = \sum_{[\xi]\in\widehat \G} d_\xi \| \widehat f(\xi) \|_{\HS}^2
\end{equation}
holds true. Note, that the matrix norm involved in this expression is the Hilbert--Schmidt norm 
induced by the trace inner product, $\|\widehat f(\xi)\|_{\HS}^2 =\trace(\widehat f(\xi) \widehat f(\xi)^*)$.

Function spaces on $\G$ can be characterised via Fourier transform in terms of sequence spaces on 
$\widehat\G$. To simplify the notation, let 
\begin{equation*}
 \Sigma(\widehat \G)  = \{ \sigma : \widehat \G\ni \xi  \mapsto \sigma(\xi)\in  \C^{d_\xi\times d_\xi} \}
\end{equation*}
denote the set of all sequences of matrices of appropriate dimensions. Then we define
\begin{equation*}
   \ell^p(\widehat \G) = \{ \sigma\in \Sigma(\widehat \G) :      \sum_{[\xi]\in\widehat\G} d_\xi \| \sigma(\xi)\|_{S_p}^p < \infty \}
\end{equation*}
in terms of Schatten p-norms of matrices. The Parseval identity implies that the Fourier transform 
$\mathcal F$ is unitary between $L^2(\G)$ and $\ell^2(\widehat\G)$
and by straightforward estimates $\mathcal F$ maps $L^1(\G)$ to $\ell^\infty(\widehat\G)$ with inverse mapping $\ell^1(\widehat\G)\to L^\infty(\G)$. 
With this notation, $\ell^\infty(\widehat\G)$ is the multiplier algebra for $\ell^2(\widehat\G)$. This 
explains in particular the appearance of the operator-norm in the symbol estimate \eqref{eq:symbeq} and 
the Hilbert--Schmidt norm in \eqref{eq:Parseval}. 

We note that there exists another version
of $\ell^{p}$-spaces on $\Gh$ based on fixing the Hilbert-Schmidt norm of $\sigma$ and
varying the powers of $d_{\xi}$, for which we refer to
\cite[Section 10.3.3]{ruzhansky+turunen-book}. For a discussion on the relation between
the spaces $\ell^p(\widehat \G)$ and Schatten classes of operators on $L^2(\G)$ we refer to
\cite{DR13}.

As $\mathcal H_\xi$ are minimal bi-invariant subspaces of $L^2(\G)$, they are eigenspaces of all 
bi-invariant operators and, in particular, of the (negative)
 Laplacian (Casimir element)
$\mathcal L_\G$ on $\G$. In the following we denote by $\langle\xi\rangle$ the corresponding 
eigenvalues for the operator $(1-\mathcal L_{\G})^{1/2}$, i.e., we have
$f - \mathcal L_\G f =\langle\xi\rangle^2 f$ for all 
$f\in\mathcal H_\xi$. 
Then the standard $L^2$-based Sobolev spaces, spaces of smooth functions and distributions are 
characterised as
\begin{align*}
f \in H^s(\G) &\qquad\Longleftrightarrow\qquad  \langle\xi\rangle^s \widehat f(\xi)  \in \ell^2(\widehat \G),\\
f \in C^\infty(\G) &\qquad\Longleftrightarrow\qquad  \forall N : \langle\xi\rangle^N \widehat f(\xi) \in \ell^2(\widehat\G),\\
f \in \mathcal D'(\G) &\qquad\Longleftrightarrow\qquad  \exists N :\langle\xi\rangle^{-N} \widehat f(\xi) \in \ell^2(\widehat\G).
\end{align*}
In the last two lines we can replace $\ell^2(\widehat\G)$ by any other $\ell^p(\widehat\G)$. We 
denote the corresponding sequences as fast decaying (giving the space $\mathfrak s(\widehat\G)$) 
and moderately growing (giving its dual $\mathfrak s'(\widehat\G)$), respectively.

\subsection{Calculus and difference operators}
\label{SEC:differences}

Difference operators acting on matrix-se\-quen\-ces are defined in terms of smooth functions
$q\in C^\infty(\G)$ vanishing in the identity 
element of the group. If $\sigma(\xi)$ is a moderately growing sequence of matrices, then 
$\mathcal F^{-1} [\sigma]\in\mathcal D'(\G)$ is a
distribution and it makes sense to define the difference operator $\Delta_{q}$ acting on $\sigma$ as
\begin{equation*}
    \Delta_{q} \sigma := \mathcal F [ q(x) \mathcal F^{-1} \sigma ].
\end{equation*}
Difference operators useful to us here
are best defined in terms of matrix coefficients of representations. For a fixed 
representation $\xi$ we denote by $\triangle_{ij}$ the 
difference operator associated to the function $(\xi(g)-\mathrm I)_{ij}$. Then a simple calculation (see 
\cite{Ruzhansky-Turunen-Wirth:arxiv}) shows that the finite Leibniz rule
\begin{equation}\label{eq:leibniz}
   \triangle_{ij} (\sigma \tau) = (\triangle_{ij}\sigma)\tau  + \sigma (\triangle_{ij}\tau) + 
   \sum_{k=1}^{d_{\xi}} (\triangle_{ik} \sigma)(\triangle_{kj}\tau)
\end{equation}
holds true. As in \cite{Ruzhansky-Turunen-Wirth:arxiv} we fix a finite selection of representations ${}_k\xi$ such that 
$({}_k\xi(x)-\mathrm I)_{ij}$ vanishes to first order in the identity element $e\in\G$ and
that $e$ is the only common zero of all these, 
\begin{equation*}
   \{e\} = \cap_k \{ x :  {}_k\xi(x)=\mathrm I\}. 
\end{equation*}
We  collect all these difference operators into a vector $\mathbb D$ and use multi-index notation 
$\mathbb D^\alpha_\xi$ for them. Similarly we write $q^\alpha(x)$ for the corresponding function
so that $\mathbb D^\alpha_\xi=\Delta_{q^{\alpha}}$. Associated 
to the difference operators we find invariant differential operators $\partial^{(\alpha)}_x$ such that the 
Taylor expansion
\begin{equation*}
   f(x) \sim  \sum_{\alpha} \frac{\partial_x^{(\alpha)} f(e)}{\alpha!} q^\alpha(x^{-1})  
\end{equation*}
holds for all smooth functions on $\G$ and in the vicinity of the identity element $e\in\G$. These 
differential operators are not powers
of first order operators, see \cite{ruzhansky+turunen-book} for some explicit formulae.

Combining all these ingredients we can characterise H\"ormander classes $\Psi^m_{1,0}(\G)$
of  pseudo-diffferential operators, defined in local coordinates,  in terms of global symbols. 
The results of \cite{Ruzhansky-Turunen-Wirth:arxiv} imply that $A\in\Psi^m_{1,0}(\G)$ if and only if
\begin{equation}\label{eq:op-qu}
   A f (x) = \sum_{[\xi]\in\widehat\G} d_\xi \trace\big( \xi(x) \sigma_A(x,\xi) \widehat f(\xi)\big)
\end{equation}
with a symbol $\sigma_A$ satisfying \eqref{eq:symbeq}. The constants in the symbolic estimates 
\eqref{eq:symbeq} define semi-norms on $\S^m_{1,0}(\G\times\widehat\G)$
and the operator quantization 
\begin{equation} \label{eq:Psi-Map}
\Op : \S^m_{1,0}(\G\times\widehat\G) \ni \sigma_A \mapsto  A\in \mathcal L(H^{s}(\G), H^{s-m}(\G))
\end{equation}
is continuous.  For later use we define slightly more general classes of symbols.

\begin{df}\label{def:2.1}
Let $0\le \rho,\delta\le 1$ and $m\in\R$. Then we denote by $\S^m_{\rho,\delta}(\G\times\widehat\G)$ 
the set of all symbols
$\sigma=\sigma(x,\xi)\in C^\infty(\G; \mathfrak s'(\widehat \G)) = C^\infty(\G)\widehat\otimes_\pi\mathfrak s'(\widehat \G)$ satisfying the symbolic estimates
\begin{equation}\label{eq:symbeq-rho-delta}
   \|\partial^\alpha_x \mathbb D_\xi^\beta \sigma_A(x,\xi)\|_{\rm op} \le C_{\alpha,\beta} \langle\xi\rangle^{m-\rho |\beta|+\delta|\alpha|}
\end{equation}
for all multi-indices $\alpha$ and $\beta$.
The set of all operators associated to such symbols via formula
\eqref{eq:op-qu}
will be denoted as  $\Op\S^m_\rd(\G\times\widehat\G)$.
\end{df}

 In fact, it can be shown
that for $\rho>\delta$ the definition above is independent of the choice of difference
operators as long as they form a so-called strongly admissible collection,
see \cite{Ruzhansky-Turunen-Wirth:arxiv}. If $\rho\leq\delta$, we consider the
class $\S^m_{\rho,\delta}(\G\times\widehat\G)$ of symbols satisfying 
\eqref{eq:symbeq-rho-delta} for a fixed collection of (strongly admissible) difference operators,
as chosen above.

We will refer to operators associated to symbols from $\S^m_{\rho,\delta}(\G\times\widehat\G)$ as
 pseudo-differential operators
of type $(\rho,\delta)$ in the sequel.%
\footnote{This should not be confused with  locally defined 
pseudo-differential operators of type $(\rho,\delta)$ for suitable parameters $\rho$ and $\delta$. 
Except for the $(1,0)$ case we do not claim that both classes coincide. However, we conjecture this.} 
If $A$ and $B$ are such operators of type $(\rho,\delta)$, $\rho>\delta$, with symbols 
$\sigma_A$ and $\sigma_B$, then their composition $AB$ is again of the same type with 
symbol given in terms of an asymptotic expansion
\begin{equation}\label{eq:calc1}
  \sigma_{AB}(x,\xi) = \sigma_A(x,\xi)\sharp\sigma_B(x,\xi) \sim \sum_{\alpha} \frac{1}{\alpha!} \big(\mathbb D_\xi^\alpha \sigma_A(x,\xi)\big)\big(\partial_x^{(\alpha)} \sigma_B(x,\xi)\big)
\end{equation}
modulo smoothing operators. Similarly, one obtains for the (formal) adjoint $A^*$ of an operator $A$
\begin{equation}\label{eq:calc2}
  \sigma_{A^*}(x,\xi)  \sim \sum_{\alpha} \frac{1}{\alpha!}\partial_x^{(\alpha)}  \mathbb D_\xi^\alpha (\sigma_A(x,\xi)^{*}).
\end{equation}
For proofs of both statements we refer to \cite[Chapter 10]{ruzhansky+turunen-book}.

\subsection{Freezing coefficients and invariant operators}
The class of left-invariant operators $\Op\S^m_{\rd}(\widehat\G)$ corresponds 
to symbols from $\S^m_{\rd}(\G\times \widehat\G)$, which are independent 
of the variable $x$. Similarly to the local calculus of pseudo-differential operators, such left-invariant 
operators can be obtained by freezing coefficients. Thus if
$A\in \Op\S^m_{\rd}(\G\times\widehat\G)$ with symbol $\sigma_A(x,\xi)$ and if $x_0\in\G$ is fixed, then $A_{x_0}$ denotes the left-invariant
operator defined by \eqref{eq:op-qu}
in terms of $\sigma_A(x_0,\xi)$. The operator $A_{x_0}$ can be described without resorting to 
symbols. Indeed, let $\phi_\epsilon\in C^\infty(\G)$ be a family of smooth functions with shrinking supports 
in balls with radius of size $\epsilon>0$ around the identity element $e\in\G$ and with 
\begin{equation*}
   \int_\G \phi_\epsilon(x)\d x = 1 ,\qquad  \lim_{\epsilon\to0} \int_\G f(x) \phi_\epsilon(x) \d x = f(e)
\end{equation*}
for any $f\in C(\G)$. This implies that $\phi_\epsilon(x)$ gives an approximate convolution identity. 

\begin{prop}
Let $0\leq \delta,\rho\leq 1$.
Let $A\in\Op\S^m_{\rd}(\G\times\widehat\G)$ and let $x_0\in\G$ be fixed. Then the associated left-invariant operator 
$A_{x_0}\in\Op\S^m_{\rd}(\widehat\G)$ is given by 
\begin{equation*}
   A_{x_0} f (x)  = \lim_{\epsilon\to0} \int_\G \phi_\epsilon(x_0^{-1} x') A [f(x(x')^{-1}\cdot)] (x') \d x'.
\end{equation*}
\end{prop}

\begin{proof} On the level of symbols the definition of $\phi_{\epsilon}$ implies
\begin{align*}
   \lim_{\epsilon\to0} \int_\G \phi_\epsilon(x_0^{-1} x') \sigma_A(x',\xi) \d x' = \sigma_A(x_0,\xi)=\sigma_{A_{x_{0}}}(\xi),
\end{align*}
and the statement follows by rewriting this in terms of associated operators (and denoting $x'=yx$)
\begin{align*}
 A_{x_0} f (x) 
 &= \lim_{\epsilon\to0}  \sum_{[\xi]\in\widehat\G} d_\xi \int_\G \phi_\epsilon(x_0^{-1}yx) \trace( \xi(yx) \sigma_A(yx, \xi) \widehat f(\xi)\xi(y^{-1}) ) \d y\notag\\&
=   \lim_{\epsilon\to0} \int_\G \phi_\epsilon(x_0^{-1} yx) A [f(y^{-1}\cdot)] (yx) \d y,
\end{align*}
since
\begin{align*}
\widehat{f(y^{-1}\cdot)}(\xi)&=\int_{\G} f(y^{-1}z)\xi(z)^{*} \d z=
\int_{\G} f(w)\xi(yw)^{*} \d w\\&= 
\int_{\G} f(w)\xi(w)^{*} \d w \ \xi(y)^{*}=\widehat{f}(\xi)\ \xi(y^{-1}),
\end{align*}
completing the proof.
\end{proof}

\section{Ellipticity and parameter-dependent ellipticity}
\label{SEC:ellipticities}

We recall from \cite{Ruzhansky-Turunen-Wirth:arxiv} that for
$0\leq \delta<\rho\leq 1$, the operator
$A\in\mathrm{Op}\ S^m_{\rd}(\G\times\Gh)$ is elliptic if and only if its symbol 
$\sigma_A(x,\xi)$ is invertible for all but finitely many $[\xi]\in\widehat\G$ and 
for such $\xi$ its inverse satisfies 
\begin{equation}\label{eq:ell}
   \|\sigma_A(x,\xi)^{-1} \|_{\rm op} \le C \langle\xi\rangle^{-m}.
\end{equation}
Under these assumptions it follows that 
$\sigma_A^{-1}(x,\xi):=\sigma_A(x,\xi)^{-1} \in \S^{-m}_{\rd}(\G\times\widehat\G)$ and that there exists a 
parametrix $A^\sharp\in\Op\S^{-m}_{\rd}(\G\times\widehat\G)$ of order $-m$ whose symbol
$\sigma_{A^\sharp}(x,\xi) = \sigma_A^\sharp(x,\xi)$ is given in terms of an asymptotic series in 
$\S^{-m}_{\rd}(\G\times\widehat\G)$ with principal part $(\sigma_A(x,\xi))^{-1}$. 

Indeed, we may adopt the property of the existence of a parametrix 
$A^\sharp\in\Op\S^{-m}_{\rd}(\G\times\widehat\G)$ as the definition of the ellipticity in this
context. As mentioned above, it is equivalent to the condition \eqref{eq:ell} 
for the symbol $\sigma_{A}$. For $\rho=1$ and $\delta=0$, the operator class
${\rm Op}\ S^m_{1,0}(\G\times\Gh)$ coincides with the usual H\"ormander class
$\Psi^{m}(\G)=\Psi^m_{1,0}(\G)$ (see \eqref{eq:symbeq}) and, therefore, these conditions are equivalent
to the usual notion of ellipticity in $\Psi^{m}(\G)$.
The characterisation as in \eqref{eq:ell} proved to be useful in applications: for example, it was used
in \cite{DR14}
to derive a version of the Gohberg lemma on compact Lie groups and to find bounds for the 
essential spectrum of operators in terms of their matrix symbols.

For present applications we will generalise this notion to parameter-dependent ellipticity with respect to 
spectral parameters
$\lambda$ from a sector $\Lambda\subset \C$. Let $m>0$ be strictly positive. We say that 
$\sigma_A\in\S^m_{\rd}(\G\times \widehat \G)$ is parameter-elliptic with
respect to $\Lambda$ if
\begin{equation}\label{eq:ell-cond}
   \| (\sigma_A(x,\xi) - \lambda \mathrm I) ^{-1} \|_{\rm op} \le C (|\lambda|^{1/m} + \langle\xi\rangle )^{-m}
\end{equation}
holds true uniformly in $x\in\G$ and $\lambda\in\Lambda$ and for co-finitely many $[\xi]$.  If $\Lambda$ 
consists just of a ray this corresponds
to assuming the existence of a ray of minimal growth.

\begin{thm}\label{thm:param-ell}
  Assume $\sigma_A\in\S^m_{\rd}(\G\times\Gh)$ is parameter-elliptic with respect to $\Lambda$. Then 
  \begin{enumerate}
  \item
   for any $0\leq\delta\leq 1$ and $0<\rho\leq 1$,
   the resolvent of the symbol satisfies
    \begin{equation}\label{eq:2.26}
     \| \partial_\lambda^k \mathbb D_x^\alpha \mathbb D_\xi^\beta (\sigma_A(x,\xi) -\lambda\mathrm I)^{-1} \|_{\rm op} \le C_{k,\alpha,\beta}  (|\lambda|^{1/m} + \langle\xi\rangle )^{-m(k+1)}\langle\xi\rangle^{-\rho|\beta|+\delta|\alpha|}
  \end{equation}
    uniformly in $x$ and $\lambda\in\Lambda$ and for co-finitely many $[\xi]$, and
  \item for any $0\leq \delta<\rho\leq 1$,
  there exists a parameter-dependent parametrix to $A-\lambda\mathrm I$ with symbol 
  $\sigma_A^\sharp(x,\xi,\lambda)$ satisfying
  \begin{equation}\label{eq:2.27}
     \| \partial_\lambda^k \partial_x^\alpha \mathbb D_\xi^\beta \sigma_A^\sharp(x,\xi,\lambda) \|_{\rm op} \le C_{k,\alpha,\beta}  (|\lambda|^{1/m} + \langle\xi\rangle )^{-m(k+1)}
     \langle\xi\rangle^{-\rho|\beta|+\delta|\alpha|}
  \end{equation}
  uniformly in $x$ and $\lambda\in\Lambda$ and for co-finitely many $[\xi]$.  
  \end{enumerate}
\end{thm}

\begin{proof}
The key point is to prove the first statement, the second follows by the
calculus. We split the proof of \eqref{eq:2.26} into two parts 
and prove it first for symbols of left-invariant operators.
The third part will show \eqref{eq:2.27}.

\medskip
\noindent{\sl Part 1.}  Let $\sigma_A(\xi)$ be the
symbol of a left-invariant operator of order $m>0$ and assume 
without loss of generality that $\spec \sigma_A(\xi)\cap\Lambda=\varnothing$. Then
$\sigma_A^\sharp(\xi,\lambda) = (\sigma_A(\xi) - \lambda\mathrm I)^{-1}$ is just the resolvent of the symbol  
and, since $\sigma_{A}$ is parameter-elliptic with respect to $\Lambda$, we have that
\begin{equation}\label{eq:2.29}
   \| \sigma_A^\sharp(\xi,\lambda) \|_{\rm op} \le C (|\lambda|^{1/m} + \langle\xi\rangle )^{-m}
\end{equation}
holds true for all $\xi$ and all $\lambda\in\Lambda$. Our aim is to prove the corresponding estimates for 
higher order differences  and for derivatives with respect to $\lambda$. 

\noindent{\sl Step 1.1.} We consider one of the first order differences $\triangle_{ij}$
from Section \ref{SEC:differences}. The finite Leibniz 
rule~\eqref{eq:leibniz} applied to $\mathrm I = (\sigma_A-\lambda\mathrm I)^{-1}(\sigma_A-\lambda\mathrm I)$ 
implies that
\begin{equation*}
 0 = (\triangle_{ij} \sigma_A^\sharp(\xi,\lambda ) ) (\sigma_A(\xi)-\lambda\mathrm I) + \sigma_A^\sharp(\xi,\lambda) (\triangle_{ij} \sigma_A(\xi)) 
 + \sum_k (\triangle_{ik} \sigma_A^\sharp(\xi,\lambda ) )(\triangle_{kj} \sigma_A(\xi))
\end{equation*}
(using that constants are annihilated by taking differences). We treat 
\begin{equation*}
  \triangle_{ij} \sigma_A^\sharp(\xi,\lambda ) +  \sum_k (\triangle_{ik} \sigma_A^\sharp(\xi,\lambda ) )(\triangle_{kj} \sigma_A(\xi)) \sigma_A^\sharp(\xi,\lambda) 
\\  = - \sigma_A^\sharp(\xi,\lambda) (\triangle_{ij} \sigma_A(\xi)) \sigma_A^\sharp(\xi,\lambda)
\end{equation*}
as a linear equation for $\triangle_{ij}\sigma_A^\sharp(\xi,\lambda)$. 
Indeed, regarding these equations as equations for the whole family 
$\{\triangle_{ij}\sigma_A^\sharp(\xi,\lambda)\}_{i,j}$ at once allows
us to treat the family of second terms (i.e. the sums)
on the left hand side as the perturbation of the family of the first terms.
More precisely,
by assumption \eqref{eq:2.29}, the `big' coefficient block-matrix of the left-hand side is a perturbation of the identity of size 
$\langle\xi\rangle^{m-\rho}(|\lambda|^{1/m} + \langle\xi\rangle )^{-m}
\lesssim \langle\xi\rangle^{-\rho}$ and thus uniformly invertible for large $\xi$. Therefore, $\triangle_{ij}\sigma_A^\sharp(\xi,\lambda)$ is
of the size of the right-hand side, i.e., we obtain
\begin{align}\label{eq:2:}
  \| \triangle_{ij} \sigma_A^\sharp(\xi,\lambda ) \|_{\rm op}& \le C \langle\xi\rangle^{m-\rho}  (|\lambda|^{1/m} + \langle\xi\rangle )^{-2m} \notag\\&\le C' (|\lambda|^{1/m} + \langle\xi\rangle )^{-m} \langle\xi\rangle^{-\rho}.
\end{align}

\noindent{\sl Step 1.2.} We proceed by induction over the order of differences and corresponding higher order Leibniz rules. They yield linear equations for $\mathbb D^\beta_\xi\sigma_A^\sharp(\xi,\lambda)$ in terms of already estimated quantities $\mathbb D_\xi^\gamma\sigma_A^\sharp(\xi,\lambda)$, $\gamma<\beta$, and known estimates for $\sigma_A(\xi)$. The coefficient matrix of this linear equation is the same,
the right hand side contains only more terms. Terms with the worst estimates contain all differences applied to $\sigma_A$ and no
differences applied to $\sigma_A^\sharp$. Hence,
they can be estimated by $(|\lambda|^{1/m}+  \langle\xi\rangle)^{-2m} 
\langle\xi\rangle^{m-\rho|\beta|}$
and thus similarly to above we obtain
\begin{equation*}
\|\mathbb D_\xi^\beta\sigma_A^\sharp(\xi,\lambda)\|_{\rm op} \le C (|\lambda|^{1/m} + \langle\xi\rangle )^{-m}\langle\xi\rangle^{-\rho|\beta|}.
\end{equation*}

\noindent{\sl Step 1.3.} Derivatives with respect to $\lambda$ are treated by means of the resolvent identity
\begin{equation*}
   \partial_\lambda \sigma_A^\sharp(\xi,\lambda) = \partial_\lambda(\sigma_A(\xi)-\lambda\mathrm I)^{-1} = (\sigma_A(\xi)-\lambda\mathrm I)^{-2} = \sigma_A^\sharp(\xi,\lambda)^2.
\end{equation*}
This implies, after combination with the Leibniz rule \eqref{eq:leibniz},
\begin{align*}
   \partial_\lambda \mathbb D_\xi^\beta \sigma_A^\sharp(\xi,\lambda)
   &= \mathbb D_\xi^\beta \partial_\lambda \sigma_A^\sharp(\xi,\lambda)
   = \mathbb D_\xi^\beta (\sigma_A^\sharp(\xi,\lambda)\sigma_A^\sharp(\xi,\lambda)) \\
  & = ( \mathbb D_\xi^\beta \sigma_A^\sharp(\xi,\lambda) ) \sigma_A^\sharp(\xi,\lambda) + \sigma_A^\sharp(\xi,\lambda)  ( \mathbb D_\xi^\beta \sigma_A^\sharp(\xi,\lambda) )  + \mathrm{l.o.t.} 
   \\ & \lesssim (|\lambda|^{1/m} + \langle\xi\rangle )^{-2m}\langle\xi\rangle^{-\rho|\beta|},
\end{align*}
and the desired statement follows again by induction over the order of derivatives. Hence, we obtain
\eqref{eq:2.26} and \eqref{eq:2.27} in the invariant case.

\medskip
\noindent{\sl Part 2.} Now we consider the general case of a symbol 
$\sigma_A(x,\xi)$ depending on both variables.
 In order to show \eqref{eq:2.26} we observe that above considerations are uniform in $x$ so that the statement for $|\alpha|=0$ is already proven. Applying the Leibniz rule to $\mathrm I = (\sigma_A-\lambda \mathrm I)^{-1} (\sigma_A-\lambda \mathrm I)$ yields
 for first order $x$-derivatives,
 \begin{equation*}
    \partial_j  (\sigma_A(x,\xi)-\lambda \mathrm I)^{-1}  =  - (\sigma_A(x,\xi)-\lambda \mathrm I)^{-1} (\partial_j \sigma_A(x,\xi))  (\sigma_A(x,\xi)-\lambda \mathrm I)^{-1},
 \end{equation*}
 and similarly also expressions for higher order derivatives of the resolvent in terms of already known terms. Applying difference operators to 
these combined with the finite Leibniz rule for differences yields (recursively) the estimates. 
 This proves
 \begin{equation*}
    \| \partial_x^\alpha \mathbb D_\xi^\beta (\sigma_A(x,\xi)-\lambda\mathrm I)^{-1}\|_{\rm op} \le C_\alpha (|\lambda|^{1/m}+ \langle\xi\rangle)^{-m}\langle\xi\rangle^{-\rho|\beta|+\delta|\alpha|}
 \end{equation*} 
 for all multi-indices $\alpha$ and $\beta$. Together with the idea of Step 1.3 this proves  \eqref{eq:2.26}. 

\medskip
\noindent{\sl Part 3.}
In order to prove \eqref{eq:2.27} we follow the usual parametrix construction in terms of asymptotic series. We restrict the detailed consideration to left parametrices and construct a sequence $\sigma_j(x,\xi,\lambda)$ starting with
\begin{equation*}
  \sigma_0(x,\xi,\lambda) = (\sigma_A(x,\xi)-\lambda\mathrm I)^{-1}, 
\end{equation*}
which satisfies the estimates \eqref{eq:2.26}, and define
\begin{equation}\label{eq:4.2}
   \sigma_{j+1} (x,\xi,\lambda) = -  \sum_{\begin{smallmatrix}{|\alpha|+k=j+1}\\{0\le k\le j} \end{smallmatrix}}
   \frac{1}{\alpha!} (\mathbb D_\xi^\alpha \sigma_k(x,\xi,\lambda))(\partial_x^\alpha \sigma_A(x,\xi)) \sigma_0(x,\xi,\lambda).
\end{equation}
Based on the Leibniz rule we can prove symbolic estimates. Plugging \eqref{eq:2.26} into \eqref{eq:4.2} yields
 \begin{equation*}
     \| \partial_\lambda^k \partial_x^\alpha \mathbb D_\xi^\beta \sigma_j(x,\xi,\lambda) \|_{\rm op} \le C_{k,\alpha,\beta,j}  (|\lambda|^{1/m} + \langle\xi\rangle )^{-m(k+1)}
     \langle\xi\rangle^{-\rho|\beta|+\delta|\alpha|-j}.
  \end{equation*}
The left parametrix $\sigma_A^\sharp(x,\xi,\lambda)$ is obtained by forming an asymptotic sum
\begin{equation*}
    \sigma_A^\sharp (x,\xi,\lambda) \sim \sum_{j} \sigma_j(x,\xi,\lambda).
\end{equation*}
Since we assumed $\rho>\delta$ the estimates are improving in the hierarchy and the desired estimate for the parametrix follows from Lemma~\ref{lem:app1} in the Appendix and the remark following it. By construction we have
$\sigma_A^\sharp\sharp \sigma_A-\mathrm I\in\S^{-\infty}(\G\times\widehat\G)$.

An analogous construction changing \eqref{eq:4.2} in the obvious way constructs a right parametrix $\sigma_A^\flat(x,\xi,\lambda)$ with $\sigma_A\sharp \sigma_A^\flat-\mathrm I\in\S^{-\infty}(\G\times\widehat\G)$. Then again by calculus
\begin{align}
    \sigma_A^\sharp - \sigma_A^\flat &=    \sigma_A^\sharp  - \sigma_A^\sharp \sharp \sigma_A\sharp \sigma_A^\flat +  \sigma_A^\sharp \sharp \sigma_A\sharp \sigma_A^\flat - \sigma_A^\flat\notag\\
  & =   \sigma_A^\sharp \sharp ( \mathrm I  - \sigma_A\sharp \sigma_A^\flat) +  ( \sigma_A^\sharp \sharp \sigma_A - \mathrm I)\sharp \sigma_A^\flat \in \S^{-\infty}(\G\times\widehat \G).
\end{align}
\end{proof}

\begin{rem} For $m=0$ a similar statement is valid. Let  $\sigma_A\in\S^0_{\rd}(\G\times\widehat\G)$ with $0\le \delta\le 1$ and $0<\rho\le1$. Let further $M:=\sup_{x,\xi} \|\sigma_A(x,\xi)\|_{\rm op}$, then 
the set of $\lambda$ with $|\lambda|\ge M$ belongs to the resolvent set of all matrices $\sigma_A(x,\xi)$.   
We replace the parameter-ellipticity by a similar condition for {\em small} $\lambda$ from the sector $\Lambda$ and assume
\begin{equation}\label{eq:pell-0}
   \|(\sigma_A(x,\xi) -\lambda\mathrm I)^{-1} \|_{\rm op} \le 
   C (1+|\lambda|)^{-1} ,\qquad \lambda\in\Lambda \cup \{|\lambda|\ge2  M\}.
\end{equation}
This assumption is only non-trivial for $|\lambda|\le 2M$ on $\Lambda$, the estimate for large $|\lambda|$ follows from the Neumann series representation of the resolvent.

Then for all multi-indices $\alpha$, $\beta$, and all $k$ we have
  \begin{equation}\label{eq:2.30}
     \| \partial_\lambda^k \partial_x^\alpha \mathbb D_\xi^\beta ( \sigma_A(x,\xi) - \lambda\mathrm I)^{-1}\|_{\rm op} \le C(1+|\lambda|)^{-1-k} \langle\xi\rangle^{-\rho|\beta|+\delta|\alpha|} 
  \end{equation}
  uniformly in $\lambda\in\Lambda \cup \{|\lambda|\ge2 M\}$
and if $\delta<\rho$ there exists a parameter-dependent parametrix satisfying the estimate
  \begin{equation}\label{eq:2.31}
     \| \partial_\lambda^k \partial_x^\alpha \mathbb D_\xi^\beta \sigma_A^\sharp(x,\xi,\lambda)\|_{\rm op} \le C(1+|\lambda|)^{-1-k} \langle\xi\rangle^{-\rho|\beta|+\delta|\alpha|}   
  \end{equation}
  uniformly in    $\lambda\in\Lambda\cup \{|\lambda|\ge2 M\}$.
  The proof goes in analogy.
\end{rem}

\begin{rem}
We can use the ellipticity of the symbol to extend the estimates from $\Lambda$ to $\Lambda_\epsilon=\Lambda\cup\{\lambda : |\lambda|<\epsilon\}$ for small $\epsilon$. This can be done similar to the previous remark based on the resolvent bound in $\lambda=0$ and the Neumann series representation of resolvents.
\end{rem}

The parametrix constructed in Theorem~\ref{thm:param-ell} represents the resolvent of the operator if it exists, in particular, resolvents are pseudo-differential operators in our setting.

\begin{cor}   
Let $0\leq \delta<\rho\leq 1$ and assume that $\sigma_A\in\S^m_{\rd}(\G\times\widehat\G)$ is parameter-elliptic with respect to some sector $\Lambda$. If $\lambda\in\Lambda$ belongs to the resolvent set of the operator $A$ (i.e.,
%in the sense that 
there exist no nonzero smooth functions $f\in C^\infty(\G)$ with $Af=\lambda f$ or $A^*f=\overline\lambda f$), then $A-\lambda\mathrm I$ is invertible on $\mathcal D'(\G)$ and the resolvent $(A-\lambda\mathrm I)^{-1}$ belongs to $\Psi^{-m}_{\rd}(\G)$.
\end{cor}
\begin{proof}
As the parametrix is a regulariser of the operator, it follows by Atkinsons' theorem that $A-\lambda \mathrm I$ is Fredholm
between Sobolev spaces $H^{s+m}(\G)$ and $H^s(\G)$ for all $s\in\R$. 
The null space of $A-\lambda\mathrm I$ is independent of $s$ and consists of functions from $C^\infty(\G)=
\bigcap_{s\in\R} H^s(\G)$. The same is true for the null space of $(A-\lambda\mathrm I)^*$. If both are trivial, the operator $A-\lambda \mathrm I$ is invertible on all $H^s(\G)$ and thus on $\mathcal D'(\G)=\bigcup_{s\in\R} H^s(\G)$. Denote $R_\lambda = (A-\lambda \mathrm I)^{-1}$ and by $(A-\lambda \mathrm I)^\sharp$ the parametrix. Then
\begin{align}
    R_\lambda - (A-\lambda \mathrm I)^\sharp &= 
     R_\lambda - R_\lambda\circ (A-\lambda\mathrm I)\circ (A-\lambda\mathrm I)^\sharp      \notag\\
   &  = R_\lambda \circ (\mathrm I - (A-\lambda\mathrm I)\circ (A-\lambda\mathrm I)^\sharp ) \; : \;\mathcal D'(\G)\to C^\infty(\G)
\end{align}
is smoothing, has a smooth Schwartz kernel and is thus element of $\Op\S^{-\infty}(\G\times\widehat\G)$.
\end{proof}

\section{Functions of symbols and operators}
\label{SEC:functions}

In this section we discuss functions of symbols and of operators.

\subsection{Functions of invariant operators}

Functions of left-invariant operators are naturally defined in terms of the spectral calculus. Later we will show that under natural assumptions on the function the resulting operators are again pseudo-differential.  

To be precise, let $A\in\Op\S^m_{\rd}(\widehat \G)$ be left-invariant and parameter-elliptic with respect to some sector $\Lambda\subset\C$. 
 We denote $\Lambda_\epsilon:=\Lambda\cup U_\epsilon$, $U_\epsilon=\{\lambda: |\lambda|\le\epsilon\}$, and assume that the function $F$ is  holomorphic in $\C\setminus\Lambda_\epsilon$ and continuous on its closure. In particular it may  have branch points in $0$ and $\infty$. Let $\Gamma=\partial\Lambda_\epsilon$ be the  (oriented) contour encircling the sector
$\Lambda_\epsilon$. We assume decay of $F$ along $\Gamma$. Then similar to the treatment of Seeley for complex powers of classical
parameter-elliptic operators, \cite{Seeley}, we define the function $F(A)$ in terms of contour 
integrals but for our global symbols on $\G\times\widehat\G$ instead of homogeneous local symbolic components. Thus, we define $B=F(A)$ in terms of the symbol $\sigma_B(\xi)$ given by 
\begin{equation}\label{eq:sigmaB}
   \sigma_B(\xi) = \frac1{2\pi\i} \oint_\Gamma F(\lambda) \sigma_A^\sharp(\xi,\lambda) \d\lambda.
\end{equation}
Theorem~\ref{thm:functions-inv} will show that $\sigma_B\in \S^{ms}_{\rd}(\widehat\G)$ provided $F$ decays on $\Gamma$
like $|\lambda|^s$ for some $s<0$.

Note, that for each fixed $\xi$ the contour can be altered and replaced by a closed contour around the spectrum of the matrix $\sigma_A(\xi)$ and thus $\sigma_B(\xi)=F(\sigma_A(\xi))$ in the sense of usual matrix spectral calculus.
As $\mathcal H_\xi$ are finite dimensional invariant subspaces of $A$ with $A$ acting as multiplication by $\sigma_A(\xi)$ on $\mathcal H_\xi$ (identified with the space of matrix coefficients), it is thus clear that the operator $F(A)$ coincides with the operator defined by the spectral calculus and based on the same cut in the complex plane. 

Of particular importance are complex powers $F(\lambda)=\lambda^s$, $s\in\C$. If $\sigma_A$ is parameter-elliptic with respect to $\Lambda$ and $\Re s<0$, combining the decays in $\lambda$,
 the above integral converges by Theorem~\ref{thm:param-ell}
and defines a family of operators $B_s$ which is holomorphic in $s$ and satisfies the group property $B_{s+t} = B_s B_t$. Among other things, we will show that all these $B_s$ are pseudo-differential of the same type as $A$ and of appropriate orders, making them in particular all elliptic.

\subsection{Functions of symbols}  

For calculus reasons it is important to show that certain functions of symbols define again symbols. As an example, we will treat positive square roots of positive symbols and use it to show a version of 
the G\aa{}rding inequality.

\begin{thm}\label{thm:functions-inv}
Let $0\leq\delta\leq 1$ and $0<\rho\leq 1$.
Let $\sigma_A\in\S^m_{\rd}(\G\times \widehat\G)$, $m\ge0$, be parameter-elliptic with respect to a sector $\Lambda$ and assume that $F$ is analytic  in $\C\setminus\Lambda_\epsilon$ and satisfies
\begin{equation}\label{eq:F-cond}
   | F(\lambda)| \le C |\lambda|^{s}
\end{equation}
for some $s<0$. Then
\begin{equation*} 
   \sigma_B(x,\xi) = \frac1{2\pi\i} \oint_\Gamma F(\lambda) (\sigma_A(x,\xi) - \lambda\mathrm I)^{-1} \d\lambda,
\end{equation*}
defines a symbol $\sigma_B\in \S^{ms}_{\rd}(\G\times\widehat\G)$.
\end{thm}

We can write $F(\sigma_{A}):=\sigma_{B}$, so that 
$F(\sigma_{A})\in \S^{ms}_{\rd}(\G\times\widehat\G)$.

\begin{rem}\label{REM:fsposs}
Given Theorem \ref{thm:functions-inv}, by the calculus \eqref{eq:calc1}
we can also define functions of symbols for functions which are growing
at most polynomially and show that they satisfy symbolic estimates again.

Indeed, let $\sigma_{A}$ and $F$ be as in
Theorem \ref{thm:functions-inv} but assume now that \eqref{eq:F-cond} 
holds with $s\geq 0$.
Writing $F(\lambda)=F(\lambda)\lambda^{-[s]-1}\lambda^{[s]+1}$, we can define
$$
F(\sigma_{A}):=\widetilde{F}(\sigma_{A})\sigma_{A}^{[s]+1},
$$ 
where
$\widetilde{F}(\lambda):=F(\lambda)\lambda^{-[s]-1}$ satisfies 
\eqref{eq:F-cond} with $\widetilde{s}=s-[s]-1<0.$ Consequently, by 
Theorem \ref{thm:functions-inv} and the calculus \eqref{eq:calc1}
we conclude that $F(\sigma_{A})\in \S^{ms}_{\rd}(\G\times\widehat\G)$. 

The particular choice of factorising $F$ does not influence the resulting symbol. 
If we write $F(\lambda)=F_1(\lambda) \lambda^{k_1}= F_2(\lambda) \lambda^{k_2}$
for different  integers $k_1$ and $k_2$ larger than $[s]$
and define $F(\sigma_A)$ through both formulae, then $F(\sigma_A)$
is still uniquely defined. If $k_1>k_2$, then $F_1(\sigma_A(x,\xi))\sigma_A(x,\xi)^{k_1-k_2} = F_2(\sigma_A(x,\xi))$ for each $x$ and $\xi$ due to Cauchy integral theorem and the functional calculus of matrices.
\end{rem}

\begin{rem}
Arbitrary complex powers of parameter-elliptic symbols are again elliptic. If $\sigma_A\in\S^m_{\rd}(\G\times\widehat\G)$ is parameter-elliptic with respect to some sector, then by the previous remark $\sigma_A^s\in\S^{m\Re s}_{\rd}(\G\times\widehat\G)$ as well as 
$(\sigma_A^s)^{-1} = \sigma_A^{-s} \in\S^{-m\Re s}_{\rd}(\G\times\widehat\G)$.
\end{rem}

\begin{proof}[Proof of Theorem \ref{thm:functions-inv}]
As $\|(\sigma_A(x,\xi) - \lambda\mathrm I)^{-1}\|_{\rm op}\le (1+|\lambda|^{1/m} + \langle\xi\rangle)^{-m}\lesssim  \lambda^{-1}$ the contour integral converges for
$s<0$ and $\sigma_B(x,\xi)$ is well-defined. We now show the symbolic estimates.

We restrict considerations for $m>0$. For $m=0$ the proof is analogous and uses the estimate \eqref{eq:2.30} instead.
We split the proof into two parts and consider first the particular case $s\in(-1,0)$. Then the elementary estimate
\begin{equation*}
  \int_0^\infty (1+r^{1/m}+\langle\xi\rangle)^{-m}(1+r)^{s} \d r 
  \le \langle\xi\rangle^{ms} \int_0^\infty (1+\tilde r^{1/m})^{-m}  \tilde r^{s}\d\tilde r
  \le C \langle\xi\rangle^{ms}
\end{equation*}
is based on the substitution $r=(1+\langle\xi\rangle)^m \tilde r$ and the integrability of the remaining integral is ensured for large $\tilde r$ by $s<0$ and for small $\tilde r$ by $s>-1$. This estimate can be directly applied to the contour integral. 
Indeed, for all fixed $x\in\G$ and $[\xi]\in\widehat\G$ the condition of parameter-ellipticity
yields
\begin{align*}
    \|\sigma_B(x,\xi)\|_{\rm op} &\le \frac1{2\pi} \oint_\Gamma | F(\lambda) | \, \|(\sigma_A(x,\xi)-\lambda\mathrm I)^{-1}\|_{\rm op} \, |\d\lambda|\\
     &\le C \oint_\Gamma |\lambda|^{s}  (|\lambda|^{1/m} + \langle\xi\rangle)^{-m}  |\d\lambda|
     \le C' \langle\xi\rangle^{ms}
\end{align*}
with all appearing constants independent of $x$ and $\xi$. Furthermore, for all multi-indices $\alpha$ and $\beta$ we obtain (by linearity and from the definition of improper Riemann integrals)
\begin{equation*}
\partial_x^\alpha \mathbb D^\beta_\xi \sigma_B(x,\xi) = \frac1{2\pi\i} \oint_\Gamma F(\lambda) \partial_x^\alpha \mathbb D^\beta_\xi ( \sigma_A(x,\xi)-\lambda\mathrm I)^{-1} \d\lambda,
\end{equation*}
so that the first statement of Theorem~\ref{thm:param-ell} and the above argumentation imply the desired estimates.

It remains to treat the general case. Let now $s<0$ be arbitrary. Then we find a function $G$ analytic in $\C\setminus\Lambda_\epsilon$ such that $F(\lambda)= G(\lambda)^{[-s]+1}$, $[-s]$ the integer part of $-s$. Then
$|G(\lambda)|\le |\lambda|^{s/(1+[-s])}$ has an exponent in the range $(-1,0)$ and 
 \begin{equation*} 
   \sigma_C(x,\xi) = \frac1{2\pi\i} \oint_\Gamma G(\lambda) (\sigma_A(x,\xi) - \lambda\mathrm I)^{-1} \d\lambda,
\end{equation*}
defines a symbol from $\S^{ms/(1+[-s])}_{\rd}(\G\times\widehat\G)$ by the previous case. Furthermore, by the spectral calculus of matrices, $\sigma_B(x,\xi)=(\sigma_C(x,\xi))^{1+[-s]}$ and hence $\sigma_B\in \S^{ms}_{\rd}(\G\times\widehat\G)$.
\end{proof}

In the following we use the standard cut $\R_-=\{\zeta\in\R : \zeta\le0\}$ for defining the complex logarithm $\log$ on $\C\setminus\R_-$. In particular, we can define for any matrix $\sigma_A(x,\xi)$ with eigenvalues separated from $\R_-$ the matrix function $\log\sigma_A(x,\xi)$
by the usual spectral calculus. We will use this in particular for positive matrices.

\begin{thm}\label{cor:3.3}
Let $0\leq\delta\leq 1$ and $0<\rho\leq 1$.
Assume $\sigma_A\in\S^m_{\rd}(\G\times\widehat\G)$, $m\ge0$, is positive definite,
invertible, and satisfies 
\begin{equation}\label{EQ:ell2}
\|\sigma_A(x,\xi)^{-1} \|_{\rm op} \le C \langle\xi\rangle^{-m}
\end{equation}
for all $x$ and for all but finitely many $\xi$. Then $\sigma_A(x,\xi)$ is parameter-elliptic
with respect to $\R_-$. Furthermore, for any number $s\in\C$,
\begin{equation*}
    \sigma_B(x,\xi) :=  \sigma_A(x,\xi) ^ s = \exp( s  \log \sigma_A(x,\xi)) 
\end{equation*}
defines a symbol $\sigma_B\in\S^{m'}_{\rd}(\G\times\widehat\G)$, $m'=\Re (ms)$.
\end{thm}

\begin{proof}
First we show parameter-ellipticity. From ellipticity \eqref{eq:ell} we know that the spectrum of the positive matrix $\sigma_A(x,\xi)$ satisfies
\begin{equation*}
 \langle\xi\rangle^{-m}   \spec \sigma_A(x,\xi) \subset [c,C] \subset\R_+
\end{equation*}
for some (fixed) constants $c$ and $C$ independent of $x$ and $\xi$. Since the matrix $\sigma_A(x,\xi)$ is normal, the operator norm of the resolvent $(\sigma_A(x,\xi)-\lambda\mathrm I)^{-1}$ is given by the distance of $\lambda$ to the spectrum and hence
\begin{equation*}
   \| (\sigma_A(x,\xi)-\lambda\mathrm I)^{-1} \|_{\rm op} \le (c \langle\xi\rangle^m +|\lambda|)^{-1} \le C' (|\lambda|^{1/m} + \langle\xi\rangle)^{-1/m}
\end{equation*}
for any $\lambda\in\R_-$.

It remains to prove the symbol property. For $\Re s<0$ we can directly apply the Theorem~\ref{thm:functions-inv}, while for $\Re s\ge0$ we use the 
semigroup property combined with Leibniz rule. Indeed, 
if $\Re s>0$ we find an integer $k$ such that $\Re s - k<0$ and therefore $(\sigma_A(x,\xi))^{s-k}\in \S^{m(\Re s-k)}_{\rd}(\G\times\widehat\G)$ by Theorem \ref{thm:functions-inv}.
But from that we conclude $\sigma_A(x,\xi)^s = \sigma_A(x,\xi)^{s-k} \sigma_A(x,\xi)^k \in 
\S^{m'}_{\rd}(\G\times\widehat\G)$ for $m'=m\Re  s$.
The semigroup property holds due to the positivity of the matrices.
\end{proof}

In the sequel, we will use the particular case of Theorem~\ref{cor:3.3}
applied to
square roots of symbols. Here we note that condition \eqref{EQ:ell2} for
$\delta<\rho$ is equivalent
to the ellipticity of the operator $A$ in view of \eqref{eq:ell}.

\begin{cor}\label{thm:roots} 
Let $0\leq\delta<\rho\leq 1$ and let $m\geq 0$.
Assume $\sigma_A \in \S^{2m}_{\rd}(\G\times\widehat\G)$ 
is elliptic (i.e. \eqref{EQ:ell2} holds) and satisfies $\sigma_A(x,\xi)>0$ 
for all $x$ and $\xi$. Then
the square root
\begin{equation*}
   \sigma_B(x,\xi) = \sqrt{ \sigma_A(x,\xi) } 
\end{equation*}
in the sense of positive matrices is a symbol satisfying 
$\sigma_B\in\S^{m}_{\rd} (\G\times\widehat\G)$.
\end{cor}

\subsection{Functions of operators} 

So far we considered functions of symbols, which for the case of invariant operators correspond to symbols of functions of operators defined in terms of the spectral calculus. We can also reformulate Theoreom~\ref{thm:functions-inv} for approximate functions of operators defined in terms of the parameter-dependent parametrix $\sigma_A^\sharp(x,\xi,\lambda)$ of a parameter-elliptic symbol $\sigma_A=\sigma_A(x,\xi)\in\S^m_{\rd}(\G\times\widehat\G)$ with respect to some sector $\Lambda$. Assume that $F$ is analytic in $\C\setminus\Lambda_\epsilon$ and satisfies \eqref{eq:F-cond}. Then  
\begin{equation}\label{eq:4.4}
    \sigma_{F(A)}(x,\xi) := \frac1{2\pi\i} \oint_\Gamma F(\lambda) \sigma_A^\sharp(x,\xi,\lambda) \d\lambda
 \end{equation}
 converges and defines a matrix-valued function $\sigma_{F(A)}(x,\xi)$ on $\G\times \widehat\G$. Under the stronger assumption that the exponent $s\in(-1,0)$ we can apply the proof of the first part of Theorem~\ref{thm:functions-inv} and obtain
 
\begin{cor}\label{COR:functions-ops}
Let $0\leq\delta<\rho\leq 1$.
Let $\sigma_A\in\S^m_{\rd}(\G\times\widehat\G)$ be parameter-elliptic with respect to $\Lambda$ and let
 $\sigma_A^\sharp(x,\xi,\lambda)$ be the corresponding parameter-dependent parametrix. 
 Let $F$ be analytic in $\C\setminus\Lambda_\epsilon$ satisfying \eqref{eq:F-cond} for some $s\in(-1,0)$ and define the operator $F(A)$ by the symbol \eqref{eq:4.4}.
Then $\sigma_{F(A)}\in\S^{ms}_{\rd}(\G\times\widehat\G)$.
\end{cor}

Again we want to extend this to all negative $s$. For this we establish an approximate functional calculus for the operators $F(A)$.
First, we denote the set of all parameter-dependent symbols $\sigma(x,\xi,\lambda)$ satisfying 
\begin{equation}\label{eq:lambda-smooth}
 \|\partial_t^k \partial_x^\alpha \mathbb D_\xi^\beta \sigma(x,\xi,\lambda)\|_{\rm op} \le C_{j,k,\alpha,\beta} (|\lambda|^{1/m} +\langle\xi\rangle)^{-m(k+1)}  \langle\xi\rangle^{-j}, \qquad \lambda\in\Lambda,
\end{equation}
 for some $m>0$ and all $j\in\mathbb N$ as $\S^{-\infty}(\Lambda)$. They are smoothing parameter-dependent symbols in our sense and by the same integration argument as in Theorem~\ref{thm:functions-inv} we obtain
 
 \begin{cor}\label{COR:smoothing}
 Assume that the parameter-dependent symbol $\sigma\in \S^{-\infty}(\Lambda)$ is smoothing,  that $F$ is analytic in $\C\setminus\Lambda_\epsilon$ 
 and satisfies \eqref{eq:F-cond}. Then the symbol
 \begin{equation*}
     \widetilde\sigma(x,\xi) := \frac{1}{2\pi \i} \oint_\Gamma F(\lambda) \sigma(x,\xi,\lambda)\d \lambda
 \end{equation*}
 is smoothing, $\widetilde\sigma \in \S^{-\infty}(\G\times\widehat\G)$.
 \end{cor}

This allows to treat smoothing parameter-dependent remainders. Combined with the following Lemma we conclude that the statement of Corollary~\ref{COR:functions-ops} is valid without restricting $s$.
 
 \begin{lem}
 Let $0\leq\delta<\rho\leq 1$ and let $\sigma_A\in\S^m_{\rd}(\G\times\widehat\G)$, $m>0$, be parameter-elliptic with respect to $\Lambda$.
 \begin{enumerate}
 \item
 Then the parameter-dependent parametrix $\sigma_A^\sharp(x,\xi,\lambda)$ satisfies the approximate resolvent identities
 \begin{equation}\label{eq:4.6}
    \sigma_A^\sharp(x,\xi,\lambda)-\sigma_A^\sharp(x,\xi,\mu) = (\lambda-\mu) \sigma_A^\sharp(x,\xi,\lambda)\sharp \sigma_A^\sharp(x,\xi,\mu)\mod \S^{-\infty}(\Lambda)
 \end{equation}
 and 
 \begin{equation}\label{eq:4.7}
     \partial_\lambda \sigma_A^\sharp (x,\xi,\lambda) = \sigma_A^\sharp(x,\xi,\lambda) \sharp \sigma_A^\sharp(x,\xi,\lambda) \mod \S^{-\infty}(\Lambda).
 \end{equation}
 
 \item 
 Assume further that $F(\lambda) = G(\lambda)^k$ for some analytic function  $G$  on $\C\setminus\Lambda_\epsilon$ which satisfies \eqref{eq:F-cond} with $s\in(-1,0)$. Then the identity
 \begin{equation}
     \sigma_{F(A)} (x,\xi) = \sigma_{G(A)}(x,\xi)^{\sharp k} = \underbrace{ \sigma_{G(A)}\sharp  \sigma_{G(A)}\sharp\cdots\sharp \sigma_{G(A)}}_{\text{$k$ times}}(x,\xi)
 \end{equation}
 is valid modulo $\S^{-\infty}(\G\times\widehat\G)$.
 \end{enumerate}
 \end{lem}
 
 \begin{proof}
 {\sl Part 1.} We show the first one. By construction, 
 \begin{equation}
 \mathrm I = \sigma_A^\sharp(x,\xi,\lambda)\sharp(\sigma_A(x,\xi)-\lambda\mathrm I)=  \sigma_A^\sharp(x,\xi,\lambda)\sharp(\sigma_A(x,\xi)-\mu\mathrm I) - (\lambda-\mu)\sigma_A^\sharp(x,\xi,\lambda) 
\end{equation} 
modulo $\S^{-\infty}(\Lambda)$ such that composing this with $\sharp\sigma_A^\sharp(x,\xi,\mu)$ from the right proves \eqref{eq:4.6}. The equality \eqref{eq:4.7}  follows from \eqref{eq:4.6} combined with the differentiability of the parametrix with respect to $\lambda$.  

\medskip
\noindent{\sl Part 2.} We consider the case $k=2$, the situation for higher powers is similar. For fixed $x$ and $\xi$ we can find two disjoint admissible paths $\Gamma_1$ and $\Gamma_2$ with $\Gamma_2$ on the $\Lambda$-side of $\Gamma_1$.  Then by definition we have
\begin{align*}
    \sigma_{G(A)} \sharp \sigma_{G(A)}  (x,\xi) &= -\frac1{4\pi^2} \oint_{\Gamma_1} \oint_{\Gamma_2} G(\lambda) G(\mu) \sigma_A^\sharp(x,\xi,\lambda)\sharp \sigma_A^\sharp(x,\xi,\mu)\d\mu \d\lambda   \\
    &= -\frac1{4\pi^2} \oint_{\Gamma_1} \oint_{\Gamma_2} \frac{ G(\lambda) G(\mu) }{\lambda-\mu} \big( \sigma_A^\sharp(x,\xi,\lambda)- \sigma_A^\sharp(x,\xi,\mu) \big)\d\mu \d\lambda\\
    &= \frac1{2\pi\i} \oint_{\Gamma_1} ( G(\lambda))^2   \sigma_A^\sharp(x,\xi,\lambda)\d\lambda 
    = \sigma_{F(A)}(x,\xi)
\end{align*}
modulo $\S^{-\infty}(\Lambda)$ and based on Cauchy integral formula (and the fact that the paths are oriented clockwise around the region of holomorphy of $G$)
\begin{equation*}
    \frac1{2\pi\i}\oint_{\Gamma_2} \frac{G(\mu)}{\lambda-\mu} \d\mu= G(\lambda), \qquad     \frac1{2\pi\i}\oint_{\Gamma_1} \frac{G(\lambda)}{\lambda-\mu} \d\lambda=0.
\end{equation*}
 \end{proof}

% 
%The analogue of Remark \ref{REM:fsposs}
%is also valid in the operator setting of Corollary \ref{COR:functions-ops-2}
%for functions $F$ growing at most polynomially.
%
\section{Sobolev boundedness of operators with symbols in 
$\S^{m}_{\rd}(\G\times\widehat\G)$}
\label{SEC:Sobolev}

As an application of the symbolic calculus we show that operators with symbols
in $\S^{0}_{\rd}(\G\times\widehat\G)$ with
$0\leq \delta<\rho\leq 1$ are bounded on $L^{2}(\G)$ with the subsequent 
boundedness on Sobolev spaces.

\begin{thm}\label{THM:L2}
Let $0\leq \delta<\rho\leq 1$ and let $A$ be an operator with symbol
in $\S^{m}_{\rd}(\G\times\widehat\G)$. Then
$A$ is a bounded from $H^{s}(\G)$ to $H^{s-m}(\G)$ for any $s\in\R$.
\end{thm}
\begin{proof}
It is enough to prove the statement for $m=s=0$ from which the general
case follows  by standard arguments using the calculus \eqref{eq:calc1}.
Thus, we will show that an operator $A$ with symbol
$\sigma_{A}\in \S^{0}_{\rd}(\G\times\widehat\G)$ is bounded on $L^{2}(\G)$.
We will prove it in several steps.

First we observe that if $\sigma_{A}\in \S^{-M}_{\rd}(\G\times\widehat\G)$
with a sufficiently large $M$ then $A$ is bounded on $L^{2}(\G)$.
This follows from the fact that if $M>\delta ([\dim\G/2]+1)$, 
then we have $\|\partial_{x}^{\alpha}\sigma_{A}(x,\xi)\|_{\rm op}\leq C$
for all $|\alpha|\leq [\dim\G/2]+1$, so that $A$ is bounded on $L^{2}(\G)$
by \cite[Theorem 10.5.5]{ruzhansky+turunen-book}.

Now, by induction, suppose that operators with symbols
in the classes $\S^{-\epsilon}_{\rd}(\G\times\widehat\G)$
are bounded on $L^{2}(\G)$ for all $\epsilon>\epsilon_{0}>0$,
and let $\sigma_{A}\in \S^{-\epsilon_{0}}_{\rd}(\G\times\widehat\G)$.
Then we can estimate
$$
\|Au\|_{L^{2}}^{2}=(u,A^{*}Au)_{L^{2}}\leq \|u\|_{L^{2}} 
\|A^{*}Au\|_{L^{2}}\leq C\|u\|_{L^{2}}^{2}
$$
since, by the calculus, the operator
$A^{*}A\in \Op\S^{-2\epsilon_{0}}_{\rd}(\G\times\widehat\G)$
is bounded on $L^{2}(\G)$ by the induction hypothesis.

Finally, let $\sigma_{A}\in \S^{0}_{\rd}(\G\times\widehat\G)$.
Then we have, in particular, that for some $M>0$ we have
$\|\sigma_{A}(x,\xi)\|_{\rm op}\leq M-1$ for all $x$ and $\xi$.
Then the matrix $M^{2} I_{d_{\xi}}-\sigma_{A}(x,\xi)^{*}\sigma_{A}(x,\xi)$ is
positive, and hence the matrix
$$
\sigma_{B}(x,\xi):=(M^{2} I_{d_{\xi}}-\sigma_{A}(x,\xi)^{*}\sigma_{A}(x,\xi))^{1/2},
$$
where $I_{d_{\xi}}\in \C^{d_{\xi}\times d_{\xi}}$ is the identity matrix,
belongs to $\S^{0}_{\rd}(\G\times\widehat\G)$ by 
Corollary \ref{thm:roots}.  From the relation
$\sigma_{B}^{*}\sigma_{B}=M^{2} -\sigma_{A}^{*}\sigma_{A}$
it follows by the calculus that
$M^{2}-A^{*}A=B^{*}B$ modulo $\Op\S^{-(\rho-\delta)}_{\rd}(\G\times\widehat\G)$, 
i.e.
$$
M^{2}-A^{*}A=B^{*}B+R \; \textrm{ with some } R\in 
\Op \S^{-(\rho-\delta)}_{\rd}(\G\times\widehat\G).
$$
Consequently, we can estimate
$$
\|Au\|_{L^{2}}^{2}=(u,A^{*}Au)_{L^{2}}\leq M^{2}\|u\|_{L^{2}}^2 -\|Bu\|_{L^{2}}^{2}-
(u,Ru)_{L^{2}}\leq C\|u\|_{L^{2}}^{2}
$$
since $\delta<\rho$, and so $R$ is bounded on $L^{2}(\G)$ by the previous step.
\end{proof}

\section{G\aa{}rding inequality}
\label{SEC:Garding}

As an application of the results in the previous section
we will give a short proof of G\r{a}rding's 
inequality for global symbols on a compact Lie group. 
The proof is entirely based on the symbolic calculus.
In particular, for $\rho=1$ and $\delta=0$, 
this yields G\r{a}rding's 
inequality for the H\"ormander class $\Psi^2_{1,0}(\G)$ which can be defined
either by localisations or by \eqref{eq:symbeq-rho-delta}, which are equivalent
in view of the explanations in Section \ref{SEC:differences}.
Therefore, we can use the matrix symbolic calculus for its proof.

\begin{cor}[G\aa{}rding's inequality on $\G$]\label{cor:garding}
Let $0\leq \delta<\rho\leq 1$.
Let $A\in\mathrm{Op} \, \S^{2(\rho-\delta)}_{\rd}(\G\times\Gh)$ be elliptic and such that 
$\sigma_A(x,\xi)\ge 0$ for all $x$ and co-finitely many $\xi$. 
Then there are constants $c_1,c_2>0$
such that for any function $f\in H^{\rho-\delta}(\G)$ the inequality
\begin{equation*}
   \Re (Af, f)_{L^2} \ge c_1 \|f\|_{H^{\rho-\delta}}^{2}  - c_2 \|f\|_{L^2}^{2}
\end{equation*}
holds true.
\end{cor}

\begin{proof}
We can change $\sigma_A(x,\xi)$  for finitely many $\xi$ by adding a smoothing operator. Therefore, we may assume without loss of generality that $\sigma_A(x,\xi)>0$
for all $x\in\G$ and $[\xi]\in\widehat\G$. Our aim is to find an elliptic symbol $\sigma_B\in\S^{\rho-\delta}_{\rd}(\G\times\widehat\G)$
such that for the corresponding operator $B$ we have
\begin{equation}\label{eq:BCond}
   C = \Re A  - B^* B  \in \mathrm{Op}\, \S^{0}_{\rd}(\G\times\Gh).
\end{equation}
This implies
\begin{equation*}
 \Re   (Af, f)_{L^2} = \| Bf \|^2_{L^2} + \Re (C f,f)_{L^2}  \ge 
 c_1 \|f\|_{H^{\rho-\delta}}^{2} - c_2 \|f\|_{L^2}^{2},
\end{equation*}
which follows by Cauchy--Schwarz' inequality combined with the mapping 
properties of pseudo-differential operators between Sobolev spaces
in Theorem \ref{THM:L2}.

It remains to determine $\sigma_B(x,\xi)$. By \eqref{eq:calc1} and \eqref{eq:calc2} equation \eqref{eq:BCond} is equivalent, modulo terms from $\S^0_{\rd}(\G\times\widehat\G)$, to
\begin{align*}
   \sigma_{\Re A}(x,\xi) &= \sigma_{B^*}(x,\xi)\sigma_B(x,\xi) + \sum_{|\alpha|=1} \big(\mathbb D_\xi^\alpha \sigma_{B^*}(x,\xi) \big) \big(\partial_x^{(\alpha)} \sigma_B(x,\xi)\big) \notag\\ 
   &=  \sigma_{B}(x,\xi)^*\sigma_B(x,\xi)  \notag \\
   &\quad  +\sum_{|\alpha|=1} 
    \big(\partial_x^{(\alpha)}  \mathbb D_\xi^\alpha (\sigma_{B}(x,\xi)^*) \big) \sigma_B(x,\xi)
   +
   \big(\mathbb D_\xi^\alpha (\sigma_{B}(x,\xi)^*) \big) \big(\partial_x^{(\alpha)} \sigma_B(x,\xi)\big) 
   \notag\\
   &= \sigma_B(x,\xi)^*\sigma_B(x,\xi) +\sum_{|\alpha|=1}  \partial_x^{(\alpha)}\big( \big(\mathbb D_\xi^\alpha (\sigma_B(x,\xi)^*) \big)\sigma_B(x,\xi)
\big).
\end{align*} 
Furthermore,  since $\sigma_A(x,\xi)>0$ is self-adjoint, 
\begin{align}\label{eq:3.14}
   \sigma_{\Re A}(x,\xi) = \sigma_A(x,\xi) +  \frac12 \sum_{|\alpha|=1} \partial_x^{(\alpha)} \mathbb D_\xi^\alpha \sigma_A(x,\xi)
\end{align}
modulo $\S^0_{\rd}(\G\times\widehat\G)$. Making use of \eqref{eq:app1}, \eqref{eq:app11} and \eqref{eq:app12} in the Appendix, we see that the appearing sum in
\eqref{eq:3.14} yields a skew-symmetric matrix. Similarly,
\begin{equation*}
  \bigg( \sum_{|\alpha|=1}  \partial_x^{(\alpha)}\big( \big(\mathbb D_\xi^\alpha (\sigma_B(x,\xi)^*) \big)\sigma_B(x,\xi)
\big) \bigg)^* = 
 - \sum_{|\alpha|=1} \partial_x^{(\alpha)} \big(\sigma_B(x,\xi)^* \mathbb D_\xi^{\alpha} \sigma_B(x,\xi)\big),
\end{equation*}
so that comparing symmetric and skew-symmetric parts gives the two equations,
again modulo $\S^{0}_{\rd}$,
\begin{align*}
  & \sigma_A(x,\xi) = \sigma_B(x,\xi)^*\sigma_B(x,\xi) \notag \\
  & { } \qquad +\frac12 \sum_{|\alpha|=1}  \partial_x^{(\alpha)} \big(\mathbb D_\xi^\alpha (\sigma_B(x,\xi)^*) \big)\sigma_B(x,\xi)-(\sigma_B(x,\xi)^*) \mathbb D_\xi^\alpha \sigma_B(x,\xi)\big),\\
  &  \sum_{|\alpha|=1} \partial_x^{(\alpha)} \mathbb D_\xi^\alpha \sigma_A(x,\xi) = \sum_{|\alpha|=1}  \partial_x^{(\alpha)} \mathbb D_\xi^\alpha \big(\sigma_B(x,\xi)^* \sigma_B(x,\xi)
\big).
\end{align*}
We use the ansatz $\sigma_B(x,\xi) = \sqrt{\sigma_A(x,\xi)} + \sigma_0(x,\xi)$ with
still unknown matrices $\sigma_0=\sigma_0(x,\xi)\in \S^{0}_{\rd}(\G\times\widehat\G)$. Then the second equation is clearly true, so it remains to satisfy the first one. This yields
\begin{equation}\label{EQ:ansatz}
  0 =  \sigma_0^* \sqrt{\sigma_A} + \sqrt{\sigma_A} \sigma_0 + \frac12 \sum_{|\alpha|=1} \partial_x^{(\alpha)}\big( (\mathbb D_\xi^\alpha \sqrt{\sigma_A})  \sqrt{\sigma_A}
  - \sqrt{\sigma_A} (\mathbb D_\xi^\alpha \sqrt{\sigma_A}) \big).
\end{equation}
This equation has a symmetric solution $\sigma_0^*=\sigma_0$ and for this it rewrites as  a (uniquely solvable) Sylvester type linear equation. Its  `big' coefficient block-matrix 
$\sqrt{\sigma_A(x,\xi)}\otimes \mathrm I + \mathrm I\otimes \sqrt{\sigma_A(x,\xi)}$ satisfies 
\begin{align}
0 &\not \in  \spec( \sqrt{\sigma_A(x,\xi)}\otimes \mathrm I + \mathrm I\otimes \sqrt{\sigma_A(x,\xi)}) %\notag\\
%&\quad 
= \{ \sqrt\mu + \sqrt\nu : \mu,\nu\in\spec \sigma_A(x,\xi) \}.
\end{align}
Furthermore, as $\sqrt{\sigma_A(x,\xi)}>0$ by construction, it is self-adjoint and in particular normal. Hence, the same is true for the Kronecker products and we can
estimate the operator norm of the inverse by looking at its eigenvalues. By ellipticity of $\sigma_A$ we know, that there are constants $c$ and $C$ such that
\begin{equation*}
   \langle\xi\rangle^{-2(\rho-\delta)} \spec \sigma_A(x,\xi) \subset  [c,C]\subset\R_+
\end{equation*}
and in consequence
\begin{equation*}
  \| (  \sqrt{\sigma_A(x,\xi)}\otimes \mathrm I + \mathrm I\otimes \sqrt{\sigma_A(x,\xi)} )^{-1} \|_{\rm op} \le \frac1{2\sqrt c} \langle\xi\rangle^{-(\rho-\delta)}.  
\end{equation*}
This implies $\|\sigma_0(x,\xi)\|_{\rm op} \le C'$. To show that $\sigma_0\in\S^0_{\rd}(\G\times\widehat\G)$ we apply 
$\partial_x^\alpha\mathbb D_\xi^\beta$ to \eqref{EQ:ansatz} and observe that this gives a Sylvester equation with the same `big' coefficient matrix 
\begin{equation*}
 \big(\partial_x^\alpha\mathbb D_\xi^\beta  \sigma_0(x,\xi)\big) \sqrt{ \sigma_A(x,\xi) } + \sqrt{\sigma_A(x,\xi)}\big(\partial_x^\alpha\mathbb D_\xi^\beta \sigma_0(x,\xi) \big) = R_{\alpha,\beta} (x,\xi) 
\end{equation*}
for the `unknown' $\partial_x^\alpha\mathbb D_\xi^\beta  \sigma_0(x,\xi)$, and by induction we see that 
\begin{equation*}
 \|R_{\alpha,\beta}(x,\xi)\|_{\rm op} \le C_{\alpha,\beta} \langle\xi\rangle^{(\rho-\delta)+\delta|\alpha|-\rho|\beta|}.
\end{equation*} 
This proves $\sigma_0\in\S^0_{\rd}(\G\times\widehat\G)$.
\end{proof}

The above statement extends by pseudo-differential calculus to symbols of arbitrary positive order. Our proof follows that of \cite{Taylor:BOOK-pseudos}.

\begin{cor}[G\r{a}rding's inequality on $\G$] 
Let $0\le\delta<\rho\le1$ and $m>0$. Let $A\in\Op\S^{2m}_{\rho,\delta}(\G\times\widehat\G)$ be elliptic and such that $\sigma_A(x,\xi)\ge0$ for all $x$ and co-finitely many $\xi$. Then there are constants $c_1,c_2>0$ such that for any function $f\in H^{m}(\G)$ the inequality
$$
   \Re (Af,f)_{L^2} \ge c_1 \| f\|^2_{H^m} - c_2 \| f\|^2_{L^2}
$$
holds true.
\end{cor}
\begin{proof}
We consider the operator $\tilde A = (\mathrm I-\Delta)^{-m} \circ\Re A\circ (\mathrm I-\Delta)^{-m}$. Due to
Theorem~\ref{cor:3.3}, formulas \eqref{eq:3.14} and \eqref{eq:calc1}  this operator belongs to $\Psi^{0}_{\rho,\delta}(\G)$.
Moreover, modulo $\S^{-(\rho-\delta)}_{\rho,\delta}(\G)$, for some constant $c>0$ its matrix symbol satisfies 
\begin{equation}\label{eq:6.5}
  \sigma_{\tilde A}(x,\xi) \ge c \mathrm I
\end{equation}   
for all $x$ and co-finitely many $\xi$ due to the assumed ellipticity, and due to the relation 
\eqref{eq:3.14}.
We may alter it by adding a smoothing operator to satisfy \eqref{eq:6.5} for all $x$ and $\xi$ and define $\sigma_{B_0}(x,\xi) = \sqrt{ \sigma_{\tilde A}(x,\xi)}$. Then by construction $\sigma_{B_0}\in\S^0_{\rho,\delta}(\G\times\widehat\G)$ and the corresponding operators satisfy
\begin{equation}
	\tilde A - B_0^* B_0 \in\Psi^{-(\rho-\delta)}_{\rd}(\G).
\end{equation} 
As in the previous proof we want to improve remainders and construct a sequence 
operators $B_k \in \Psi^{-k(\rho-\delta)}_\rd(\G)$, $k\in\mathbb N$, such that
\begin{equation}
R_k =   \tilde A - (B_0+\cdots+B_{k-1})^* (B_0+\cdots + B_{k-1}) \in \Psi^{-k(\rho-\delta)}_{\rd}(\G).
\end{equation} 
For this to be satisfied for all $k$ it is sufficient that
\begin{equation}
    R_{k} - B_0^* B_k - B_k^* B_0 \in \Psi^{-k(\rho-\delta)}_{\rd}(\G),
\end{equation}
or on the level of symbols,
\begin{equation}
   \sigma_{R_k}(x,\xi) = \sigma_{B_0}(x,\xi)^*\sigma_{B_k}(x,\xi) +\sigma_{B_k}(x,\xi)^*\sigma_{B_0}(x,\xi)
\end{equation}
modulo $\S^{-(k+1)(\rho-\delta)}_{\rd}(\G\times\widehat\G)$.
In analogy to \eqref{EQ:ansatz} this is a Sylvester type equation for $\sigma_{B_k}$.
As $R_k$ is self-adjoint by construction, its symbol $\sigma_{R_k}(x,\xi)$ is self-adjoint modulo $\S^{-(k+1)(\rho-\delta)}_{\rd}(\G\times\widehat\G)$. As $\sigma_{B_0}(x,\xi)$ is also self-adjoint and satisfies \eqref{eq:6.5}, we find a unique self-adjoint solution $\sigma_{B_k}(x,\xi) = \sigma_{B_k}(x,\xi)^*$ to this equation. Again as in the previous proof it follows that 
$\sigma_{B_k}\in\S^{-k(\rho-\delta)}_\rd(\G\times\widehat\G)$.

Let now $k> (\rho-\delta)^{-1} m$ and $B= B_0+\cdots+B_k\in \Psi^0_\rd(\G)$.
Then by construction $\tilde A- B^*B \in\Psi^{-m}_{\rd}(\G)$ and thus with $f = (\mathrm I-\Delta)^{m} g$ 
we obtain the desired statement
\begin{equation}
 \Re (Af,f)_{L^2} = (\tilde Ag,g)_{L^2} \ge c_1 \|g\|^2_{L^2} - c_2 \|g\|^2_{H^{-m}}
 = c_1 \|f\|^2_{H^m} - c_2 \|f\|^2_{L^2}.
\end{equation}
This completes the proof.
\end{proof}

\section{Appendices}
\label{sec:p}
\label{SEC:appendix}

In this section we collect several technical results used throughout the paper.

\subsection{Some formulas for difference and differential operators}\label{sec:A1} 
%For convenience we collect some useful formulae for difference operators and their associated differential 
%operators. 

For a multi-index $\alpha$ we denote by $\overline\alpha$ the associated multi-index such that $\mathbb D^{\overline\alpha}$ is the difference operator obtained from $\mathbb D^\alpha$ by replacing all elementary differences $\mathbb D_{ij}$ by $\mathbb D_{ji}$. Then
\begin{equation}\label{eq:app1}
  ( \mathbb D^{\alpha} \sigma)^* =  \mathbb D^{\overline\alpha}\sigma^*.
\end{equation}
Indeed, we have, more generally,
\begin{equation}\label{EQ:app2}
(\Delta_{q}\sigma)^{*}=\Delta_{q_{0}}\sigma^{*}, \textrm{ with }
q_{0}(x)=\overline{q(x^{-1})}.
\end{equation}
Writing $\sigma(\xi)=\widehat{k}(\xi)$ for some $k$,
this follows from
$$
(\Delta_{q}\sigma(\xi))^{*}=\int_{\G}\overline{q(x)} \overline{k(x)} \xi(x) \d x=
\int_{\G}\overline{q(x^{-1})} \overline{k(x^{-1})} \xi(x)^{*} \d x=
(\Delta_{q_{0}}\sigma^{*})(\xi)
$$
and 
$$
\sigma(\xi)^{*}=\int_{\G} \overline{k(x)}\xi(x) \d x=
\int_{\G} \overline{k(x^{-1})}\xi(x^{-1}) \d x=\widehat{k_{0}}(\xi),
$$
with $k_{0}(x)=\overline{k(x^{-1})}.$
Consequently,  since $\Delta_{ij}=\Delta_{q}$ with $q(x)=(\eta(x)-I)_{ij}$ for a
representation $\eta$, the equality \eqref{eq:app1} follows from \eqref{EQ:app2} since
then $q_{0}(x)=\overline{(\eta(x^{-1})-I)_{ij}}=\overline{(\eta(x)^{*}-I)_{ij}}=(\eta(x)-I)_{ji}.$

The map $\alpha\mapsto \overline\alpha$ is a bijection of multi-indices of fixed order, thus when summing over $\alpha$, the formula \eqref{eq:app1} yields
\begin{equation}\label{eq:app11}
  \sum_{|\alpha|=k}   ( \mathbb D^{\alpha} \sigma)^* =  \sum_{|\alpha|=k}  \mathbb D^{\alpha}\sigma^*. 
\end{equation}
Similarly, we relate the operators $\partial^{(\alpha)}$ and $\partial^{(\overline\alpha)}$. For this we note that for $i\ne j$ the properties of unitary matrix representations imply that the matrix entries satisfy $\overline{\xi_{ji}(x)} = {\xi_{ij}(x^{-1})}$. We will use this relation to distribute the derivatives as `symmetric' as possible in
\begin{equation*}
   f(x) = f(e) +  \sum_{|\alpha|=1}  q^\alpha(x) {\partial^{(\alpha)}f(e)}  +\mathcal O(x^2),
\end{equation*}
and we can find {\em first} order differential operators $\partial^{(\alpha)}$ in such a way that
\begin{equation}\label{eq:app12}
  \overline{\partial^{(\alpha)}f} = - \partial^{(\overline\alpha)} \overline f,\qquad |\alpha|=1.
\end{equation}
These operators are useful in asymptotic expansions when dealing with adjoint operators.

The operators $\partial^{(\alpha)}$ for $|\alpha|=1$ are first order differential operators annihilating constants. This directly follows from the definition in terms of Taylor's formula. Thus, they obey the first order Leibniz rule
\begin{equation*}
    \partial^{(\alpha)}(fg) = (   \partial^{(\alpha)}f)g +   f(\partial^{(\alpha)} g),\qquad |\alpha|=1.
\end{equation*} 
Higher order Leibniz rules are more involved as the higher order operators are not just powers of first order operators.

\subsection{Asymptotic summation}
In this section we will provide a proof of the asymptotic summation formulas in the symbol classes $\S_{\rho,\delta}^m(\G\times\widehat\G)$ for all
$\rho>\delta$ and for parameter-dependent symbols in the sense of this paper. We will start with a technical lemma involving bounds for difference operators. We denote by
$\varkappa = \lceil (\dim \G)/2 \rceil$ the smallest integer larger than half the dimension of $\G$.

\begin{lem}\label{lem:5.1}
Assume $\sigma\in\Sigma(\widehat\G)$ satisfies $\sup_{[\xi]} \langle\xi\rangle^{-m}\|\sigma(\xi)\|_{\rm op}=:\mu<\infty$ for some $m\in\mathbb R$. Then for any difference operator $\Delta_q$ defined in terms of a function $q\in C^\infty(\G)$ the estimate
\begin{equation*}
   \|\Delta_q \sigma(\xi)\|_{\rm op} \le C \mu \|q\|_{C^{\varkappa+\lceil |m| \rceil}(\G)} \langle\xi\rangle^m
\end{equation*}
holds true with a constant $C$ independent of $\sigma$ and $q$.
\end{lem}
\begin{proof}
Assume that $\sigma=\widehat R$ for some distribution $R\in\mathcal D'(\G)$ and let $A : f \mapsto f*R$ be the corresponding left-invariant operator. Then clearly 
$$
\sup_{[\xi]} \langle\xi\rangle^{-m} \|\sigma(\xi)\|_{\rm op} = \|A\|_{H^m\to L^2}.
$$
We introduce a further parameter $z\in\G$ and consider the operators
\begin{equation*}
  A_z : f \mapsto \int_\G f(y) q(y^{-1}z) R(y^{-1} x) \d y
\end{equation*}
defined for $f\in C^\infty(\G)$. Then $A_z f = A \circ M_{q_z} f$ with $M_{q_z}$ the multiplication operator with the function $q_z : y\mapsto q(y^{-1}z)$ and $f(x)  \mapsto A_x f(x)$ corresponds to the convolution with the kernel $qR$ and thus the symbol $\Delta_q\sigma$.
Therefore, it suffices to estimate
\begin{align*}
   \| A_x f(x)\|_2^2 &= \int_\G |A_x f(x)|^2\d x \le \int_\G \sup_z |A_z f(x)|^2 \d x\notag\\
   &\le C \sum_{|\alpha|\le \varkappa} \int_\G\int_\G | \partial_z^\alpha (  A\circ M_{q_z} ) f(x)|^2 \d z \d x \notag\\
   &=  C \sum_{|\alpha|\le \varkappa} \int_\G\int_\G |  A\circ M_{ \partial_z^\alpha q_z} f(x)|^2 \d z \d x \notag\\
   &\le C  \sum_{|\alpha|\le \varkappa} \int_\G  \|  M_{ \partial_z^\alpha q_z} \|_{H^m\to H^m}^2 \d z \|A\|^2_{H^m\to L^2} \|  f\|_{H^m}^2
  \notag\\ & \le C' \|q\|^2_{\mathfrak M_m} \|A\|_{H^m\to L^2}^2 \|f\|_{H^m}^2
\end{align*}
by the aid of Sobolev embedding theorem and with 
\begin{equation*}
   \|q\|_{\mathfrak M_m}^2 = \sum_{|\alpha|\le \varkappa}  \sup_{z\in\G}  \|  M_{ \partial_z^\alpha q_z} \|_{H^m\to H^m}^2
\end{equation*}
a corresponding multiplier norm. Hence,
\begin{equation*}
    \sup_{[\xi]}  \langle\xi\rangle^{-m}   \| \Delta_q \sigma(\xi)\|_{\rm op} = \| f\mapsto A_x f\|_{H^m\to L^2} \le \sqrt{C'} \|q\|_{\mathfrak M_m} \sup_{[\xi]} \langle\xi\rangle^{-m} \|\sigma(\xi)\|_{\rm op}.
\end{equation*}
But now the statement follows directly from the multiplier estimate $\|M_q\|_{H^m\to H^m} \le \|q\|_{C^{\lceil m\rceil}}$ for $m>0$ and by duality
$\|M_q\|_{H^{-m}\to H^{-m}} \le \|q\|_{C^{\lceil m\rceil}}$. 
\end{proof}

This will be the key statement used for remainder estimates in the asymptotic summation.  Instead of considering symbols directly, we will look at right convolution kernels and estimate them in the space
\begin{equation*}
   \mathbb X_m := \bigcap_k C^k ( \G ; H^{-m-\delta k -\varkappa} (\G)).
\end{equation*}
\begin{lem}\label{lem:5.2}
\begin{enumerate}
\item[{\rm (1)}]
Assume $\sigma\in\S^{m}_{\rd}(\G\times\widehat\G)$, then the inverse Fourier transform
of $\sigma(x,\cdot)$, i.e. $R(x,y):=\mathcal F^{-1}[\sigma(x,\cdot)](y)$, 
 satisfies $R\in\mathbb X_m$.
\item[{\rm (2)}]
If a distribution $R\in C^\infty(\G ; \mathcal D'(\G))$ belongs to $\mathbb X_m$ for some $m$, then
its Fourier transform with respect to the second argument, 
$\sigma(x,\xi) := \mathcal F [ R(x,\cdot)](\xi)$, satisfies 
\begin{equation*}
   \| \partial_x^\alpha \sigma(x,\xi) \|_{\rm op} \le C_\alpha \langle\xi\rangle^{m + \delta |\alpha| + \varkappa}.
\end{equation*}
\end{enumerate} 
\end{lem}
\begin{proof}{} (1) This is a straightforward consequence of the Fourier characterisation of Sobolev spaces in combination with $ \| \sigma(x,\xi)\|_{\rm op}^2\le \|\sigma(x,\xi)\|_{\rm \HS}^2 \le d_\xi \| \sigma(x,\xi)\|_{\rm op}^2 $ and 
\begin{equation}\label{eq:5.10}
  \sum_{[\xi]} d_\xi^2 \langle\xi\rangle^{-2\varkappa} <\infty.
\end{equation}
The latter just rephrases the well-known fact that the embedding $H^\varkappa(\G)\to L^2(\G)$ is Hilbert--Schmidt\footnote{For another explanation of this and for other relations
on compact Lie groups, see also \cite{DR13a}.}
whenever $\varkappa>\frac12\dim\G$.
Indeed, let  $\sigma\in\S^{m}_{\rd}(\G\times\widehat\G)$, then 
$$
 \langle\xi\rangle^{-2m-2\delta|\alpha|-2\varkappa}  \|  \partial_x^\alpha \sigma(x,\xi) \|_{\rm \HS}^2 \le d_\xi  \langle\xi\rangle^{-2m-2\delta|\alpha|-2\varkappa}    \|  \partial_x^\alpha \sigma(x,\xi) \|_{\rm op}^2 \le C_\alpha d_\xi \langle\xi\rangle^{-2\varkappa},
$$
which is summable by \eqref{eq:5.10} and yields the desired estimate for the Sobolev norm.

(2) If $R\in\mathbb X_m$, then using the notation $\langle\mathrm D\rangle^s$ for the pseudodifferential operator with symbol $\langle\xi\rangle^s$ and the embedding 
$L^2(\G)\to L^1(\G)$, we get
\begin{align*}
  \langle\xi\rangle^{-m-\delta|\alpha|-\varkappa}\| \partial_x^\alpha \sigma(x,\xi)\|_{\rm op}
  &\le\|   \langle\mathrm D\rangle^{-m-\delta|\alpha|-\varkappa} \partial_x^\alpha R(x,\cdot)\|_{L^1}
\\
 & \le \| \partial_x^\alpha R(x,\cdot)\|_{H^{-m-\delta|\alpha|-\varkappa}} ,
\end{align*}
and the latter is bounded uniformly in $x$ due to our assumption.
\end{proof}

\begin{lem}\label{lem:app1}
Let $\sigma_j\in\S^{m_j}_{\rd} (\G\times\widehat\G)$, $j\in\mathbb N_{0}$,
$0\le\delta<\rho\le 1$, be a family of symbols with $m_j\searrow-\infty$. Then there exists a symbol $\sigma\in\S^{m_0}_{\rd} (\G\times\widehat\G)$ such that
\begin{equation*}
  \sigma - \sum_{j=0}^{N-1}\sigma_j \in \S^{m_N}_{\rd}(\G\times\widehat\G)
\end{equation*}
for all $N\in\mathbb N_0$.
\end{lem}
\begin{proof}
Let $\psi_\epsilon\in C^\infty(\G)$ be an approximate convolution identity supported near the identity element, such that for any $s\in\mathbb R$ and all $f\in H^s(\G)$, we have
\begin{equation*}
   \lim_{\epsilon\to0}  \|  \psi_\epsilon*f  - f \|_{H^s} = 0.
\end{equation*}
Then its Fourier transform $\widehat \psi_\epsilon$ gives a family of symbols from $\S^{-\infty}(\widehat \G)$ tending to $1$ as $\epsilon\to0$. We will use $\psi_\epsilon$ to remove some regular parts from the summation in order to make it summable.

Let $\sigma_j\in \S^{m_j}_{\rd} (\G\times\widehat\G)$ and denote by $R_j$ the corresponding right convolution kernel. Then $R_j\in\mathbb X_{m_j}$ by Lemma \ref{lem:5.2},
and we can find a sequence $\epsilon_j\to0$ such that
$$
   \| \partial_x^\alpha R_j(x,\cdot) - \psi_{\epsilon_j}* \partial_x^\alpha R_j(x,\cdot) \|_{H^{-m_j-\delta|\alpha|-\varkappa}} \le 2^{-j}
$$
holds true for all $|\alpha|\le j$. Note that the same estimate holds for all Sobolev spaces with higher smoothness since the embeddings $H^{s'}(\G)\to H^s(\G)$ have norm $1$ for $s'\ge s$.

Hence, the series 
$$
   \sum_{j\ge M} R_j(x,\cdot) - \psi_{\epsilon_j} * R_j(x,\cdot)
$$
converges in $C^{k}(\G ; H^{-m_M-\delta k  - \varkappa }(\G))$ for all $k=0,1,\ldots,M$. In particular,
defining $\sigma$ as the Fourier transform of the full series starting with $M=0$, we obtain
\begin{align*}
   \sigma (x,\xi) &= \sum_{j=0}^\infty  \sigma_j(x,\xi)  (1-\widehat{ \psi_{\epsilon_j}}(\xi) ) \notag\\
   & = \sum_{j=0}^{N-1} \sigma_j(x,\xi) + \sum_{j=N}^{M-1} \sigma_j(x,\xi) - \sum_{j=0}^{M-1} \sigma_j(x,\xi)\widehat{\psi_{\epsilon_j}}(\xi) \notag\\ & \qquad+ \sum_{j\ge M} \sigma_j(x,\xi) (1-\widehat{ \psi_{\epsilon_j}}(\xi) ).
\end{align*}
We consider these sums separately. The first one clearly belongs to $\S^{m_0}_{\rd}(\G\times\widehat \G)$, the second to $\S^{m_N}_{\rd}(\G\times\widehat \G)$ and the third is smoothing.
In order to show that $\sigma$ belongs to $\S^{m_0}_\rd(\G\times\Gh)$ and is the desired asymptotic sum it is sufficient to find for all multi-indices $\alpha$ and $\beta$ a number $M$ such that the last sum satisfies
\begin{equation*}
   \bigg \| \partial_x^\alpha \mathbb D_\xi^\beta \sum_{j\ge M} \sigma_j(x,\xi) (1-\widehat{ \psi_{\epsilon_j}}(\xi) ) \bigg \|_{\rm op} \le C_{\alpha,\beta} \langle\xi\rangle^{m_{N}-\rho|\beta|+\delta|\alpha|}.
\end{equation*}
 By Lemma~\ref{lem:5.1} it suffices to find $M$ such that the estimate (with the bound depending on
 $\beta$) is valid {\em without} taking differences. Furthermore, as in the proof of Lemma~\ref{lem:5.2} (2) it is sufficient to find $M$ such that 
 \begin{equation*}
  \sum_{j\ge M} \sigma_j(x,\xi) (1-\widehat{ \psi_{\epsilon_j}}(\xi) ) \in C^{|\alpha|} (\G ; H^{-m_{N}-\delta|\alpha|+\rho|\beta|}(\G)).
 \end{equation*} 
 By construction, this happens for $M\ge |\alpha|$ in combination with $m_{N}+\delta|\alpha|-\rho|\beta|> m_M+\delta|\alpha|+\varkappa$. 
\end{proof}

\begin{rem}
If we have families of symbols $\sigma_j(x,\xi,\lambda)$ depending uniformly on a parameter $\lambda$ then the construction in the previous proof is also uniform with respect to the parameter. This can be applied to the construction of parameter-dependent parametrices, where $(|\lambda|^{1/m}+\langle\xi\rangle)^{m} \sigma_j(x,\xi,\lambda)\in\S^{-j}_{\rd}(\G\times\widehat\G)$ uniformly with respect to $\lambda\in\Lambda$.
\end{rem}

%\bibliographystyle{alphaabbr}
%\bibliography{bib_pdo-13-05-25}

\end{document}